\documentclass[12pt]{article}

\usepackage{caption}
\usepackage[top=0.75in]{geometry}
\usepackage{amsmath,amsthm,amssymb}
\usepackage{enumerate, xcolor, graphicx}
\usepackage{bbm}
\usepackage{tcolorbox}
\usepackage{algorithm}
\usepackage{hyperref}
\usepackage{courier}
\usepackage{lineno}
\usepackage{fdsymbol}
\usepackage{multirow}

\newtheorem{thm}{Theorem}
\newtheorem{lem}[thm]{Lemma}
\newtheorem{prop}[thm]{Proposition}
\newtheorem{cor}[thm]{Corollary}

\theoremstyle{definition}
\newtheorem{defi}{Definition}
\newtheorem*{rem}{Remark}
\newtheorem*{notation}{Notation}

\def\lf{\left\lfloor}   
\def\rf{\right\rfloor}
\def\lc{\left\lceil}   
\def\rc{\right\rceil}

\usepackage{tikz}
\newcommand*\circled[1]{\tikz[baseline=(char.base)]{
            \node[shape=circle,draw,inner sep=2pt] (char) {#1};}}
            
\begin{document}

\title{No-feedback Card Guessing Game: Moments and distributions under the optimal strategy \vspace{-0.5em}}
\author{Tipaluck Krityakierne$^{1}$ \and Poohrich Siriputcharoen$^{1}$ \and Thotsaporn Aek Thanatipanonda$^{2}$  \and Chaloemkiat Yapolha$^{1}$}

\date{%
\vspace{-1.5em}
    \footnotesize{$^1$Department of Mathematics, Faculty of Science, Mahidol University, Bangkok, Thailand\\%
    $^2$Science Division, Mahidol University International College, Nakhon Pathom, Thailand}
}

\vspace{-1.0em}
\maketitle
\begin{abstract}
\vspace{-0.5em}
Relying on the optimal guessing strategy recently found for a no-feedback card guessing game with $k$-time riffle shuffles, we derive an exact, closed-form formula for the expected number of correct guesses and higher moments for a $1$-time shuffle case. Our approach makes use of the fast generating function based on a recurrence relation, the method of overlapping stages, and interpolation. As for $k>1$-time shuffles, we establish the expected number of correct guesses through a self-contained combinatorial proof. The proof turns out to be the answer to an open problem listed in Krityakierne and Thanatipanonda (2022), asking  for a combinatorial interpretation of a generating function object introduced therein.\vspace{0.5em}

\end{abstract}

\section{Preliminary Discussion}

In a typical card guessing game, the player guesses the card one at a time until all cards in the deck have been guessed. 
The object of the game is to maximize the number of correct guesses. 
Depending on the rule, the player may or may not receive feedback. For the complete feedback game, in which the player sees the card after each guess, the player can adjust his strategy according to the cards he has already seen \cite{BD,KT1}. The no-feedback game is equivalent to the game that the player guesses all the cards beforehand. For a brief review on this topic, see e.g. \cite{KT2}.

The game starts with the $n$ cards ordered from 1 to $n$ before the deck is given ``riffle-shuffled'' $k$ times. The optimal guessing strategy for $k$-time shuffles with no feedback was recently established in \cite{KT2}. The leading term of the asymptotic expected number of correct guesses (for large $n$) was also found in \cite{KT2}. Building upon the optimal guessing strategy, in this work, we derive the exact, closed-form formula (in $n$) for the expectation and all higher moments for the 1-time shuffle. In order to achieve our goal for 1-time shuffle, we introduce the generating function (for each numeric $n$) that directly enumerates and counts the number of correct guesses of all possible permutations after a single shuffle in Section \ref{sec:Introduction to Card Guessing}. This is most direct, but of course slowest. We then upgrade it to fast (and faster) generating functions by setting up an appropriate recurrence relation in Section \ref{sec:fasterG}. Any versions of these generating functions can be used to compute the higher moments numerically.  Section \ref{sec:combinatorial_moments} presents a combinatorial method for calculating the moments. Although we show how to obtain a closed-form formula directly for the first two moments with the combinatorial approach, finding the expression for higher moments becomes very tedious.  Theorem \ref{Main} established in Section \ref{sec:highermoments} ensures that the formula of the higher moments can be expressed in terms of a combination of polynomials. This theorem enables us to make good use of the numeric expression from the generating function methods we developed. Using it to generate data required for interpolation, we can recover the unknown coefficients of the polynomials, and ultimately obtain the closed-form expression for all the higher moments. 

For $k>1$-time shuffles, we give a combinatorial proof for the formula of the expectation in Section  \ref{sec:k_shuffles}. This serves as a combinatorial interpretation for the generating function objects developed in \cite{KT2}, and is the answer to the open problem 1 listed in that paper. 

With the structure of the paper outlined above, our aim is to present the work that can be read from multiple points of view. 

\begin{tcolorbox}
{\bf High perspective:} demonstrate how to apply symbolic programming like Maple to rigorous mathematics research. \vspace{1em}

{\bf Medium perspective:} develop a method, step-by-step, to go from numeric to symbolic computations. \vspace{-1em}
\begin{enumerate}
\item (numeric) a slow, fast and fastest way to compute generating functions leading to the distribution of number of correct guesses for a big numeric $n$.
The three versions of generating functions have been developed in a progressive manner. The new and faster method 
is always compared to the previous method to make sure that the improved version is correct.
\item (symbolic) the fastest method for generating functions, together with the combinatorial approach and interpolation, is used to derive a closed-form formula for the moments. 
\end{enumerate}
\vspace{1em}

{\bf Specific perspective:} provide the distribution
of number of  correct guesses based on the optimal strategy.
\end{tcolorbox}

All Maple programs used to evaluate the results in this paper can be found at {\footnotesize{\url{https://thotsaporn.com/Card3.html}}}. We also provide a specific Maple command at the location where it is used throughout the paper. Look for the texts starting with {\bf$\mathtt</\hspace{-0.3em}>$}.

Let us take this opportunity to end the section with the method of overlapping stages, which will play a key role in our work when computing higher moments.

\textbf{The method of overlapping stages}

The method of overlapping stages for computing higher moments has been presented previously in \cite{Z1}.  The methodology is simple, yet very powerful. We briefly discuss it here mainly for the sake of self-containedness.

Assume a discrete random variable  $X$ can be decomposed into the sum of $n$ indicator random variables $X_1, X_2, \dots, X_n$, i.e. $X=\sum_{i=1}^n X_i$, and $X_i$ takes only value 1 or 0.

{\bf Goal:} find a formula (in $n$) of $E[X^r]$ for $r\geq1$.

Every first-year student knows how to calculate $E[X]$ from 
 \[ E[X] = \sum_i i\cdot P(X=i).  \]
 At first glance this formula above may look simple and convenient; however, when $P(X=i)$ is not available, the formula is obviously not applicable. Even if it is available, finding a closed-form expression of the sum can be very challenging (even for a binomial distribution). 

There is a better way to do this. For the first moment, the method of overlapping stages is simply the linearity of expectation: 
 \[E[X] = E[X_1+X_2+\dots+X_n] = E[X_1]+E[X_2]+\dots+E[X_n].\]
 
The higher moment can also be calculated in a similar fashion:
\[ E[X^r] = \sum_{i_1=1}^n \sum_{i_2=1}^n \dots \sum_{i_r=1}^n  E[X_{i_1}X_{i_2}\dots X_{i_r} ]. \]

For example, if $X=X_1+X_2+X_3$,
\[ E[X^2] = E[(X_1+X_2+X_3)^2]=\sum_{i=1}^3 \sum_{j=1}^3 E[X_{i}X_{j}]. \]

\newpage
 \textbf{Example 1} (Number of fixed points in the permutation)
  
 Let $X$ be a random variable of the number of fixed points in the permutation $\sigma$ of length $n$. 
 Let $\sigma_i$ be the entry in position $i$ of the permutation $\sigma$. 
 
 Then, we can decompose $\displaystyle X=\sum_{i=1}^n X_i$, where  for each $1\leq i \leq n,$
 \[ X_i = 
 \begin{cases} 1, & \text{ if } \sigma_i=i, \\ 
 0, & \text{ otherwise.} \end{cases} \]

 The first moment, i.e. the expectation, of $X_i$ is 
 \[ E[X_i] = \dfrac{(n-1)!}{n!} = \dfrac{1}{n},\]
 as we fix the position $\sigma_i=i$ and give the freedom to the other positions.
Similarly,  for any $i \neq j,$
 \[ E[X_iX_j] = \dfrac{(n-2)!}{n!} = \dfrac{1}{n(n-1)},\]
 as we fix the positions $\sigma_i=i, \sigma_j=j$ and give the other $n-2$ positions the freedom.

The expectation can be obtained by
\[  E[X] = E[X_1]+E[X_2]+\dots+E[X_n] 
= \dfrac{(n-1)!}{n!} + \dots + \dfrac{(n-1)!}{n!} = 1.\]

And for the second moment,
\begin{align*}  
E[X^2] &= E[(X_1+X_2+\dots+X_n)^2] \\
&= E[X_1^2]+E[X_2^2]+\dots+E[X_n^2]+2(E[X_1X_2]+\dots+E[X_{n-1}X_n]) \\
&= E[X_1]+E[X_2]+\dots+E[X_n]+2(E[X_1X_2]+\dots+E[X_{n-1}X_n]) \\
&= 1+2\left(\dfrac{(n-2)!}{n!} + \dots + \dfrac{(n-2)!}{n!}\right) = 1+1 = 2.
\end{align*}
Therefore, \[ Var(X) = E[X^2]-E[X]^2= 2-1^2 = 1.\]


\section{Introduction to Card Guessing}
\label{sec:Introduction to Card Guessing}
Let us start with a formal definition of a {\it riffle shuffle}.


\begin{defi}[Riffle shuffle \cite{AZ}] Gilbert-Shannon-Reeds model for riffle shuffles is performed by taking out a deck of $n$ cards initially labeled consecutively from 1 to $n$. Split the initial identity permutation into two sequences/piles (possible to have 0 cards in one of the piles) in such a way that the probability of cutting the first $t$ cards is ${\binom{n}{t}}/{2^n}$. Let $a$  ($b$) be the number of cards in the first (second) pile. To complete the riffle shuffle, interleave the cards from the two piles back in any possible way. In particular, the card from the first (second) pile will be dropped next with probability $\dfrac{a}{a+b}$ $\left(\dfrac{b}{a+b}\right)$. Since there are $\binom{n}{t}$ ways to interleave the two piles, and all these possibilities of interleaving are equally likely, each interleaving has probability ${1}/{2^n}$ to come up.
\end{defi}

Before we delve into technical details of card guessing, and even before mentioning the optimal guessing strategy, 
it is helpful to spend a moment thinking about the problem from an elementary probability perspective, e.g. a sample space, simple events, finding the mean, etc. Although simple, this concept will be served as a basis for the calculation of the $r$th moment in Section \ref{sec:combinatorial_moments}. 

{\bf Let's bring things back to basics: 1-shuffle}

{\bf - Experiment:} Suppose that a deck of $n$ cards (initially ordered from 1 to $n$) is given a riffle-shuffle once. The experiment starts from the process of splitting the deck into two piles then interleaving the piles back into a single one.  Although the identity permutation $[1,2,\dots,n]$ in the final outcome has multiplicity $n+1$, we treat each of them as different outcomes as they were split at different locations before interleaving. Thus, there are $2^n$ possible outcomes, each of these permutations (simple events) constitute the sample space of our experiment.

{\bf - Expected number of correct guesses:}  The ability to trace back to simple events allows us to use elementary concepts to find the expected number (and higher moments). For example, to find the mean, $E[X]$, all we need to do is add up the numbers of correct guesses $X(\pi_i)$ (under some guessing strategy $\cal{G}$, not necessarily optimal) of the $i$th permutation, $\pi_i$, in the sample space and then divide the sum  by the total number of $2^n$ possible outcomes: $E[X]=\dfrac{\sum_{i=1}^{2^n}X(\pi_i)}{2^n}$.

{\bf - Variance and higher moments:} Similarly, to find the $r$th moment, $E[X^r]$, we do not use the probability mass function of $X$, but instead, we average the value $[X(\pi_i)]^r$ over all permutations in the sample space: $E[X^r]=\dfrac{\sum_{i=1}^{2^n}[X(\pi_i)]^r}{2^n}$. 

{\bf - Sum of the number of correct guesses:} 
In line with $E[X]$, we use the notation $C[X]$ to denote the sum of all possible number of correct guesses (under some guessing strategy $\cal{G}$) in the sample space. That is, $C[X] = \sum_{i=1}^{2^n}X(\pi_i)$, the numerator of $E[X]$. The notation $C[X^r]$ can be defined similarly. 

We now discuss about the optimal guessing strategy for guessing the card position $i$.

\textbf{The optimal strategy $\cal{G}^*$ for card guessing with no feedback} 

For a deck of $n$ cards, the optimal strategy ${\cal G}^*$ for multiple-time shuffles ($k\geq1$), for $n$ large, was explicitly found in \cite{KT2} where one should guess the top half of the deck with sequence 
\[
\underbrace{1,\dots, 1}_{2^k-1\text{ times}}, \;\underbrace{2,\dots, 2}_{2^k\text{ times}}, \;\underbrace{3,\dots, 3}_{2^k\text{ times}}, \;\underbrace{4,\dots, 4}_{2^k\text{ times}}, \dots  
\]
and guess the bottom half in the reverse manner, i.e. 
\[
\dots, \underbrace{n-3,\dots, n-3}_{2^k\text{ times}}, \;\underbrace{n-2,\dots, n-2}_{2^k\text{ times}},\; \underbrace{n-1,\dots, n-1}_{2^k\text{ times}}, \;\underbrace{n,\dots, n}_{2^k-1\text{ times}}. 
\]

This is the unique optimal strategy for $k>1$; however, an optimal strategy for a 1-time shuffled deck is not unique. In particular,
for the card position $i$, the player can optimally choose to guess any number from the set $\mathcal{S}_i$, where
\[ \mathcal{S} =   \{1\}, \{2\}, \{2\}, \{2,3\}, \{3\}, \{3,4\}, \{4\}, \{4,5\}, \{5\}, \dots   
\;\ \;\ \;\ \;\ \;\ \;\ \;\ \;\ \text{Top half}   \]
and (in the reverse manner) \[   \dots, \{n-3,n-2\}, \{n-2\}, \{n-2,n-1\}, \{n-1\}, \{n-1\}, \{n\}   
\;\ \;\  \text{Bottom half}.  \]

While the strategy may not be optimal for small $n$, to be consistent with $k>1$, we shall stick with the following guessing strategy and refer to it as an optimal strategy ${\cal G}^*$ for the case $k=1$.
\begin{tcolorbox}
{\bf Optimal guessing strategy ${\cal G}^*$ for 1-time shuffle}
\[ 1,2,2,3,3,4,4,\dots   
\;\ \;\ \;\ \;\ \;\ \;\ \;\ \;\ \;\ \;\ \;\ \;\ \;\ \;\ \;\ \;\ \;\ \;\ \text{Top half} \]
and 
\[ \dots, n-3, n-3, n-2, n-2, n-1, n-1, n\;\ \;\ \text{Bottom half}.   \]
\end{tcolorbox}

Let us remark that the symmetry of the optimal guessing strategy ${\cal G}^*$ (mirror images of each other along the position ``half-deck'') will later play a crucial role in simplifying several calculations in this work. 

 \textbf{Example 2} (4 cards, 1-time shuffle)
Let us consider an experiment of shuffling a deck of $4$ cards once. Figure \ref{fig:sample_space} gives the complete sample space of 16 permutations. 
 \begin{figure}[h!]
    \centering
	\includegraphics[width=1\textwidth]{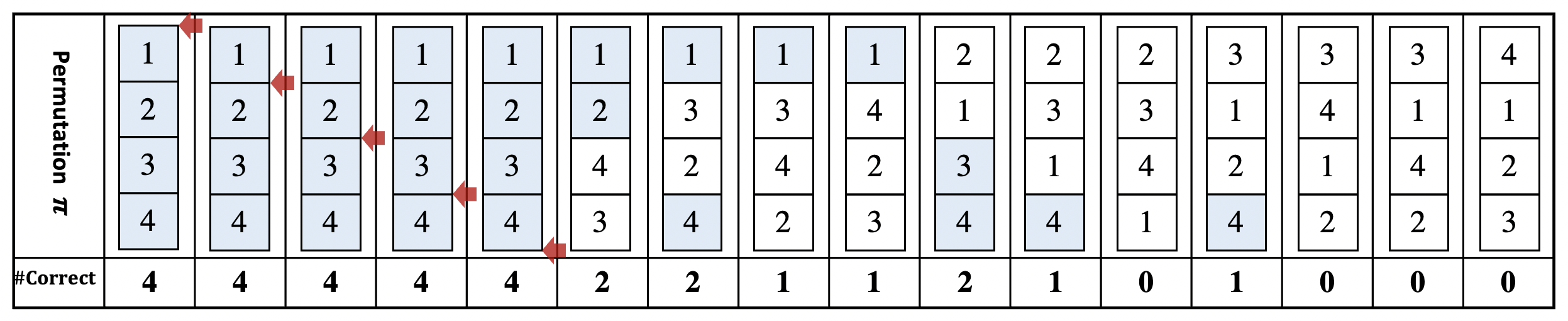} 
    \caption{Sample space for an experiment with a deck of 4 cards after one shuffle. The red arrow marks the location where the deck was split before interleaving. The blue color indicates a correct guess under the optimal strategy ${\cal{G}}^*=[1,2,3,4]$.}
    \label{fig:sample_space}
\end{figure}
 
Since there are 5 identity permutations,  in order to distinguish one permutation from another, a red arrow is used to indicate the position where the deck was split before interleaving. For example, the first identity permutation is the outcome of splitting the deck into the piles $\emptyset$ and $[1,2,3,4]$. For the other permutations, it is clear from the final permutation where the deck got split; hence, no arrow is needed. For example, the permutation $[1,3,4,2]$ is the outcome of splitting the deck into the piles $[1,2]$ and $[3,4]$ before interleaving. 

To find the expected number of correct guesses under the optimal strategy ${\cal{G}}^*=[1,2,3,4]$, we simply average the number of correct guesses from each permutation and obtain 
\[
E[X] = \frac{C[X]}{2^4}=\frac{(5\times4)+(3\times2)+(4\times1)}{16}=\frac{30}{16}=1.875.
\]

Similarly, we can compute the second moment
\[
E[X^2] = \frac{C[X^2]}{2^4}=\frac{(5\times4^2)+(3\times2^2)+(4\times1^2)}{16}=\frac{96}{16}=6.
\]

It is important to point out that in order to use the method of overlapping stages to find moments, we need to determine the number of possible permutations whose $X_i=1$, $X_iX_j=1$, etc.
(which is one of the goals of this work presented in Section \ref{sec:highermoments}).
Before we move on to the next topic, let us use this simple example to demonstrate the method to find the numerators $C[X]$ and $C[X^2]$:
\begin{align*}  
C[X] &= C[(X_1+X_2+X_3+X_4)] \\
&= 9+6+6+9=30.\\
C[X^2] &= C[(X_1+X_2+X_3+X_4)^2] \\
&= C[X_1^2]+C[X_2^2]+C[X_3^2]+C[X_4^2] \\
&+2(C[X_1X_2]+C[X_1X_3]+C[X_1X_4]+C[X_2X_3]+C[X_2X_4]+C[X_3X_4]) \\
&= \left(9+6+6+9\right)+2(6+5+6+5+5+6)=96.
\end{align*}

\newpage
{\bf The generating function by enumeration: \circled{I} A direct but slow way $\spadesuit$}

With the optimal strategy ${\cal G}^*$ in mind, the (counting) generating function can be built to keep track of the number of correct guesses for each of the $2^n$ permutations of an $n$ card deck, after one riffle shuffle: 
\begin{equation}
\label{eq:slow}
 F_{n}(q) = \sum_{i=0}^{n}a_iq^i, 
 \end{equation}
where $a_i$ denotes the number of permutations 
with $i$ correct guesses under the optimal strategy $\cal{G}^*$.

Computing the generating function using \eqref{eq:slow} would require an exhaustive list of every possible permutation before we can count the number of correct guesses one by one, which becomes very slow whenever $n$ gets large. For example, we could compute only up to $n=15$ cards for a single shuffle.
We note that, $F_{n}(1)=2^n$, the total number of possible permutations.
 
Continuing with the previous example in Figure \ref{fig:sample_space}, the generating function is
\[
F_{4}(q)
=4+4q+3q^2+5q^4.
\]

Observe that 
$C[X]=F'_{4}(q)\vert_{q=1}=30$ and $C[X^2]=\left[qF'_{4}(q)\right]'\vert_{q=1}=96$.

{\bf$\mathtt</\hspace{-0.3em}>$}Maple command for this empirical (slow) approach is {\bf\texttt{GenSlow(n,q)}}.\\
For example, try:
{\bf\texttt{GenSlow(12,q);}}

\section {Fast and faster computation of generating function for a single shuffle \label{sec:fasterG}}

\subsection{Fast and faster generating functions}
In this section, we will consider only a single shuffle ($k = 1$). This will later provide a basis for the general case ($k>1$) which is much more involved and will be discussed in detail in a subsequent section.

It can be seen from the previous example that the direct computation of generating function even for one shuffle can be time consuming. 
Luckily, there is a faster way to obtain the generating function using a recurrence relation. We first define the two sequences produced by a riffle shuffle formally, as they will play a key role in our work.  

\begin{defi}[The two sequences produced by a riffle shuffle] 
\label{def:increasing_seq}
We will call the two sequences obtained by splitting the initial identity permutation the {\it first} and {\it second sequences}, and denote them by 
$A$ and $B$, respectively.  Note that the two sequences can be specified immediately after splitting the deck, and before interleaving the cards from the two piles.
\end{defi}

From the above definition, we can see that the two sequences always take the form $A=[1,2,\dots,s-1]$ and $B=[s,s+1,\dots,n]$ for some $1\leq s\leq n+1$ (of course, when $s=1$, $A=\emptyset$, and similarly, when $s=n+1$, $B=\emptyset$).

Revisit the example in Figure \ref{fig:sample_space}. The sequences $A$ and $B$ of the second permutation are $A=[1]$, and $B=[2,3,4]$ while for the last permutation, they are  $A=[1,2,3]$, and $B=[4]$.

\textbf{Top half of the deck}

Because of  the symmetry of the optimal guessing strategy, we will first concentrate on the top half of the deck and derive a recurrence relation for it.

\begin{notation}
Let $h=\lc \dfrac{n}{2} \rc$ be the length of the half deck. Let $a_1,a_2\geq0$ be the length of the first and second sequences in the {\bf top half} of the permutation $\pi$, i.e. $a_1+a_2 = h$. Let $s$ be the starting number of the second sequence in the top half.
\end{notation}

\begin{defi}
Shuffle an $n$-card deck once. For fixed parameter values $a_1, a_2, s$, let $\Pi\left[a_1,a_2,s\right]$ be the set containing all permutations whose top half satisfies these parameters.
We denote by $G(a_1,a_2,s,q)$ a generating function that counts the number of correct guesses in the top half for permutations in $\Pi\left[a_1,a_2,s\right]$.
\end{defi}

\textbf{Example 3} Let $a_1=3, a_2=2$ and $s=7$. Then, the length of the top half is $h=a_1+a_2=5$, and $\Pi\left[3,2,7\right]$ contains those permutations whose top half consisting of card numbers  $1,2,3,7,8$ as shown in Figure \ref{fig:exampleTophalf}. There are total of $\binom{3+2}{3} = 10$ such half-permutations, each of which has its own number of correct guesses according to the optimal strategy $[1,2,2,3,3]$ at the first five positions.
Thus, $G(3,2,7,q) = 2q+4q^2+4q^3$.  

\begin{figure}[h!]
    \centering
\includegraphics[width=0.8\textwidth]{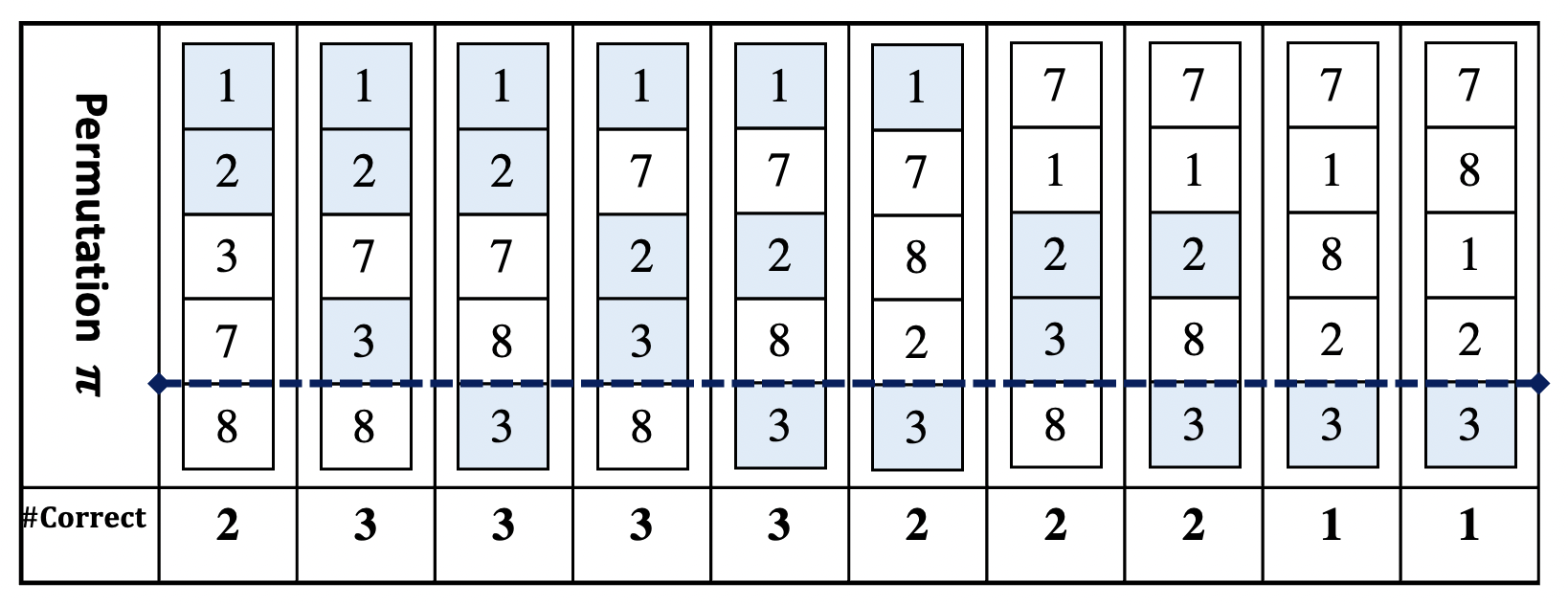}  
    \caption{Top-half permutations with parameters $a_1=3, a_2=2$ and $s=7$}
    \label{fig:exampleTophalf}
\end{figure}

\textbf{A recurrence relation for the top-half deck}

We clearly see from Figure \ref{fig:exampleTophalf} that the number of correct guesses of the top five cards is equal to the number of correct guesses of the top four cards plus 1 if the fifth card matches the best guess ``3'', or plus zero, otherwise. Thus, we can define a recurrence relation for the top-half deck: 
\[ G(3,2,7,q) = q^{1}G(2,2,7,q)+q^{0}G(3,1,7,q). \]

Generalizing the concept, a recurrence relation can be obtained by comparing the last position of the half-permutation (which will either end with the entry of the first sequence or
that of the second sequence) with the best guess $c$ of that position under the optimal strategy $\cal{G}^*$. 

We state this formally in the following proposition which also suggests a way to obtain a fast calculation method for recursively defining the generating function $G$.

\begin{prop} Let $c= \lf\dfrac{a_1+a_2}{2}\rf+1$. Then,  $c$ is the optimal guess at the position ($a_1+a_2$)th for a 1-time shuffled deck under $\cal{G}^*$, and 
\begin{equation} 
\label{fastG}
G(a_1,a_2,s,q) = q^{\delta(c,a_1)}G(a_1-1,a_2,s,q)+q^{\delta\left(c,s+a_2-1\right)}G(a_1,a_2-1,s,q),  
\end{equation}
where $G(0,0,s,q) = 1$.  $\delta(c,d)$ is the Kronecker delta function, taking value $1$ if $c=d$, and $0$ otherwise.
\end{prop}

It is clear from the proposition that the approach is general and applicable to other guessing strategies.

\textbf{Assembly the top-half generating function $G$ to the full generating function}

To get the full generating function, i.e. the number of correct guesses of the whole permutation, we multiply the generating function of the top half with that of the bottom half.

\begin{prop} Shuffle an $n$-card deck once, and let $\pi$ be a resulting permutation.  Denote by  $h= \lc\dfrac{n}{2}\rc$ the length of the top half of $\pi$. Let $a$ and $b$ be  the length of the first sequence in the top half, and the length of the second sequence in the bottom half of  $\pi$, respectively.
Then, the generating function of $\pi$ is 
\begin{equation}
\label{eq: fullG}
F(a,b,h,q):=G(a,h-a,a+(n-h-b)+1,q)\cdot G(b, n-h-b, b+(h-a)+1,q).
\end{equation}
\end{prop}

\begin{proof}
The first and second $G$'s in \eqref{eq: fullG} are used to track the number of correct guesses in the top and bottom half of $\pi$. The first $G$ follows from the discussion above. The parameters used in the second $G$ can be explained as follows. First, we mirror each entry of the bottom half about $n$, i.e. each entry $e$ in the bottom half is relabelled to $n+1-e$. Then, the entries are reversed, the first entry becoming the last, and the last entry becoming the first.
For example, the sequence $[4,11,5,12,6,13,14]$ in the bottom half with $n=14$ is rewritten to $[1,2,9,3,10,4,11]$. Moreover, the original optimal guessing strategy for the bottom half $[11,11,12,12,13,13,14]$ will be rewritten to $[1,2,2,3,3,4,4]$. 
The relabelling allows us to use the same function $G$ defined previously for the top half with the bottom half, and \eqref{eq: fullG} follows immediately.
\end{proof}

{\bf \circled{II} A fast approach for computing the generating function $\spadesuit\spadesuit$}

Summing over all possible values of parameters $a$ and $b$, we obtain a fast approach for computing the generating function for the number of correct guesses of all permutations after one shuffle. 

\begin{cor}
\label{cor: fullFn}
Let $F_n(q)$ be the generating function for the number of correct guesses of all permutations after one shuffle.  Then,
\begin{equation}   
\label{eq:fullFn_fast}
F_n(q) = \sum_{a=0}^h \sum_{b=0}^{n-h} G(a,h-a,a+(n-h-b)+1,q)\cdot G(b, n-h-b, b+(h-a)+1,q).  
\end{equation}
\end{cor}

The fact that the generating function defined in Corollary \ref{cor: fullFn} is a multiplication of two recursive functions allows us to compute the generating function $F_n(q)$ for as large as $n=1000$. Figure \ref{fig:histogram} displays histograms for the distribution of the number of correct guesses for several values of $n$.

\begin{figure}[h!]
    \centering
    \includegraphics[width=0.8\textwidth]{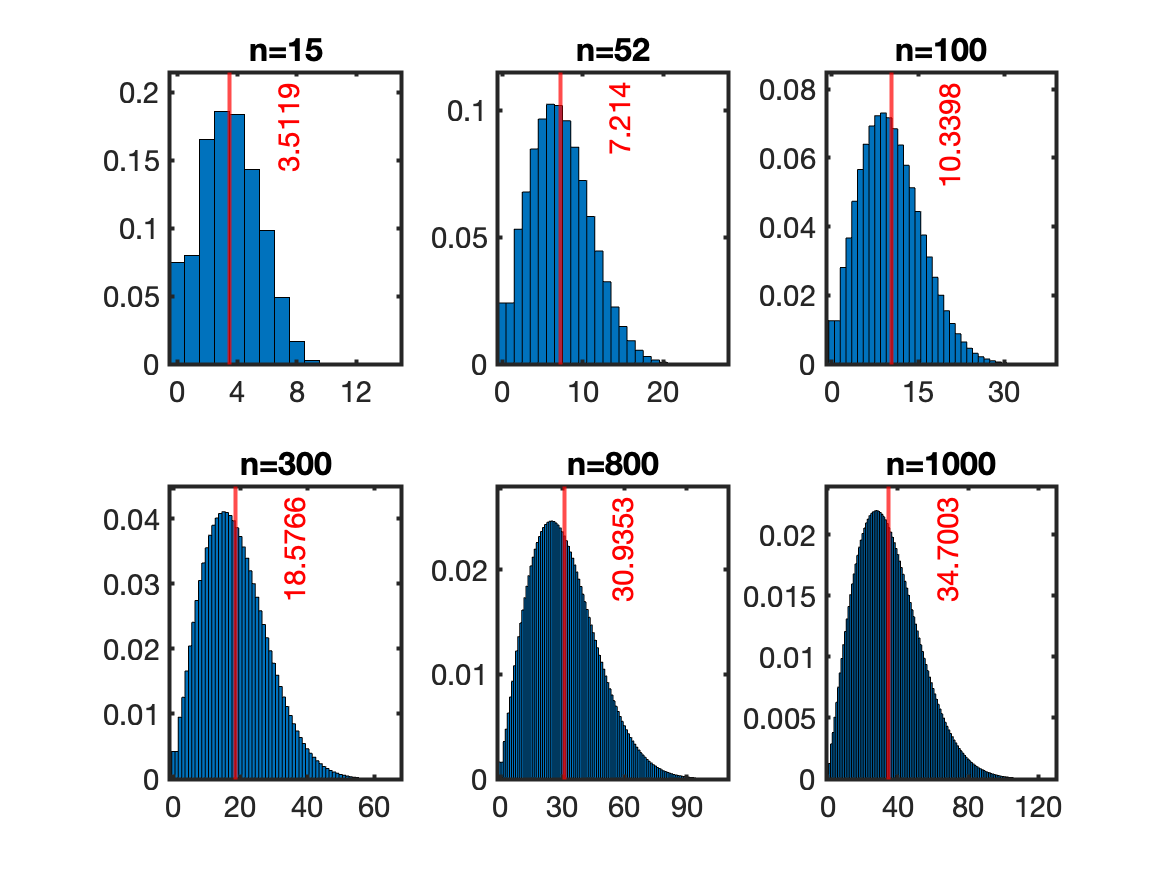}  
    \caption{Probability histograms of the number of correct guesses when $n$ varies. The red vertical line indicates the corresponding expected value $E[X]$.}
    \label{fig:histogram}
\end{figure}

It is clear that these distributions are not normal. We will discuss about the $r$th moment in the subsequent sections.

{\bf Fastest computation of generating function, finally!}

By modifying the recurrence relation of the half-generating function $G$ in \eqref{fastG}, we achieve a recurrence that can be used to compute very efficiently the generating function. Realizing the fact that there
are only two ways that the correct guesses in the top half can come from the second sequence,
we will ignore them first, and only make some modification 
at the very end of the calculation.
Thus, our new recurrence will keep track of only those correct guesses from
the first sequence. 

The two ways that the correct guesses in the top half come from the second sequence
are   

\[  [ {\color{blue}1} | {\color{red} 2},3,4,\dots,h] 
\text{ and } [\emptyset | { \color{red} 1,2},3,4,\dots,h], \]
where the vertical bar is used to separate the first and second sequences. The blue and red colors indicate the correct guesses from the first and second sequences, respectively.

To achieve an even faster calculation, the recurrence of $G$ is set as
\[ G(a_1,a_2,q) = q^{\delta(c,a_1)}G(a_1-1,a_2,q)+G(a_1,a_2-1,q),  \]
with the same base case $G(0,0,q) = 1$, and $\delta(c,d)$ is the Kronecker delta function.

As we do not have to take into account correct guesses from the second sequence, we dropped the variable $s$. This makes the computations very fast.

{\bf  \circled{III} Fastest approach $\spadesuit\spadesuit\spadesuit$}

The next step is to combine the top-half generating function (the correct guesses from the first sequence) with the bottom-half generating function (the correct guesses from the second sequence) by multiplying the two generating functions and summing over all possible values of $a$ and $b$. The resulting generating function is still different from $F_n(q)$ in \eqref{eq:fullFn_fast} due to the 4 missing permutations. Thus, we simply add back the generating functions corresponding to the number of correct guesses from those permutations that involve either (1) the second sequence in the top half:
\[  [ {\color{blue}1} | {\color{red} 2},3,4,\dots,h] 
\text{ and } [\emptyset | { \color{red} 1,2},3,4,\dots,h], \]
or (2) the first sequence in the bottom half:
\[  [ h+1,h+2,\dots,{\color{red}n-1}|{\color{blue}n}] 
\text{ and } [ h+1,h+2,\dots,{\color{red}n-1,n}|\emptyset]. \]

Each of these four half-permutations leads to the same full-deck identity permutation, $[1,2,3,\dots,n]$, each contributing 4 number of correct guesses (i.e. $q^4$).

The new and fastest calculation for the generating function is therefore
\begin{align}   
\label{fastestF}
F_n(q) &=-2q^2-2q^3+4q^4+ \sum_{a=0}^h \sum_{b=0}^{n-h} G(a,h-a,q)\cdot G(b, n-h-b,q)  \notag\\
&=-2q^2-2q^3+4q^4+ \sum_{a=0}^h G(a,h-a,q) \cdot \sum_{b=0}^{n-h} G(b, n-h-b,q),
\end{align}
where the first two terms compensate for double counting parts of each of these four permutations. For example, since the object  $[ {\color{blue}1} | {\color{red} 2},3,4,\dots,h]$ is equivalent to $[ {\color{blue}1} | {\color{red} 2},3,4,\dots,h,h+1,\dots,{\color{blue}n-1,n}]$, and those blue correct guesses have been counted already in the double series, we thus subtract $q^3$ from the final generating function.

Compared to the first two versions, this last method could generate $F_n(q)$ for $n=1000$ within 11.4 seconds, and will be used as the computation method for the moments in the following sections.

{\bf$\mathtt</\hspace{-0.3em}>$}Maple command for the fast and fastest approach are {\bf\texttt{GenFast(n,q)}} and {\bf\texttt{GenFastest(n,q)}}.\\
For example, try:
{\bf\texttt{GenFast(120,q);}} and
{\bf\texttt{GenFastest(120,q);}}

\subsection{Numerical moments from generating functions}
\label{sec:numericMoment}

For a fixed value of $r\geq0$, we can use any versions of the generating function $F_n(q)$ in \eqref{eq:slow}, \eqref{eq:fullFn_fast}, or \eqref{fastestF} to get a numeric $r$th moment for the number of correct guesses. Let $X$ be the number of correct guesses. Then,
\begin{equation}
\label{eq:numericMoment}
C[X^r]=D^{(r)}F_n(q)\vert_{q=1}, 
\end{equation}
where the operator $Dt(q):=qt'(q)$ and $D^{(r)}$ means we repeatedly apply the operator $D$, $r$ times.

By noticing that the expression $F_n(q)$ in \eqref{fastestF} can be partitioned into three parts: 

\[  
F_n(q) =\underbrace{-2q^2-2q^3+4q^4}_{\text{excess}}+
\underbrace{\sum_{a=0}^h G(a,h-a,q)}_{:=F_A(q)} \cdot
\underbrace{\sum_{b=0}^{n-h} G(b, n-h-b,q)}_{:=F_B(q)}.
\]

we can get more out of this fastest version. In particular, we obtain the following results.

\begin{cor}
\label{cor:CY_A,CZ_B} 
Let $Y_A$ be the number of correct guesses in the top half from the first sequence and $Z_B$ the number of correct guesses in the bottom half from the second sequence, respectively. Then,
\begin{align}
C[Y^r_A]&=2^{n-h}\cdot D^{(r)}F_A(q)\vert_{q=1}\\
C[Z^r_B]&=2^h\cdot D^{(r)}F_B(q)\vert_{q=1}, 
\end{align}
where $F_A(q)=\sum_{a=0}^h G(a,h-a,q)$ and $F_B(q)=\sum_{b=0}^{n-h} G(b, n-h-b,q)$. 
\end{cor}

Moreover, the relation \eqref{fastestF} also implies the independence between the number of correct guesses from the first sequence in the top and those from the second sequence in the bottom halves of the deck. We state this in the next corollary.

\begin{cor}
\label{cor:independent}
$Y_A$ and $Z_B$ are  independent of each other.
\end{cor}

\section{A combinatorial approach for calculating the $r$th moment for a one-time shuffled deck\label{sec:combinatorial_moments}}

\subsection{Three-level procedure for finding the $r$th moment}

Again, we let $X$ be a random variable representing the number of correct guesses of the whole deck after one shuffle. Our ultimate goal is to find a closed-form formula for $E[X^r]$, the $r$th moment of $X$, as a function in $n$ based on the optimal guessing strategy $\cal{G}^*$. While it is possible to achieve this by differentiating both sides of the relation \eqref{fastestF}, the calculations of closed-form expression becomes very involved. In this section, we will establish a systematic calculation framework for the  $r$th moment from the combinatorial point of view instead. 

Using the same notation as before, we let $n$ denote the number of cards, and $h= \lc n/2 \rc$ the number of cards in the top half of the deck. Moreover, since the deck is given a riffle-shuffle once, it will produce two increasing sequences according to Definition \ref{def:increasing_seq}. Recall that we called these two sequences $A$ and $B$, respectively.

Let $a$ be the length of $A$ in the top half, and $b$ the length of $B$ in the bottom half.
Since $A$ will always start with number 1, obviously $A$ consists of the numbers $[1,2,\dots,a]$ in the top half. Similarly, $B$ consists of the numbers $[n-b+1,n-b+2,\dots,n]$ in the bottom half.

As we have already seen, due to symmetry we can focus solely on the number of correct guesses in top half of the deck. 
Moreover, since there are only two permutations that the correct guesses in the top half could come from the second sequence, we again ignore them for the moment and only focus on those correct guesses from the first sequence, $A$. 

We now present the three-level procedure, a combinatorial approach for calculating the $r$th moment about the origin (raw moment). 

\begin{tcolorbox}
\textbf{High Level:} Let $Y_A$ be the number of correct guesses in the top half from the first sequence and $Z_B$ the number of correct guesses in the bottom half from the second sequence, respectively. Then, 
\begin{equation} 
E[X^r] =  E[ (Y_A+Z_B)^r]+ \frac{e(r)}{2^n},
\end{equation}
where
\begin{equation}E[ (Y_A+Z_B)^r] = \sum_{i=0}^r \binom{r}{i} E[Y_A^i]\cdot E[Z_B^{r-i}],  
\end{equation}
and $e(r)$ is given by \eqref{eq:f(q)}.
\end{tcolorbox}

The last equality holds because $ E[Y_A^{i}Z_B^{r-i}] = E[Y_A^{i}]E[Z_B^{r-i}]$ for the independent random variables $Y_A$ and $Z_B$ (see Corollary \ref{cor:independent}). The next proposition tells us how to obtain the term  $e(r)$ explicitly.

\begin{prop}
\label{prop:excess}
Let $e(r)=2^n\left(E[X^r]-E[ (Y_A+Z_B)^r]\right)$ be the excess term, which collects leftover contributions from the second sequence in the top half ($Y_B$), the first sequence in the bottom half ($Z_A$), and the cross-terms.

Then,
\begin{equation}
\label{eq:f(q)}
e(r)=4\cdot4^r-2(3^r+2^r).
\end{equation}
\end{prop}

\begin{proof}
We recall the excess term of the generating function in \eqref{fastestF} and the relationship between the generating function and the $r$th moment. In particular, let $f(q):=4q^4-2(q^3+q^2)$ and define the operator $D t(q):=qt'(q)$. Then, the result follows immediately as
\[
e(r) = D^{(r)}f(q)\vert_{q=1}=4\cdot4^r-2(3^r+2^r).
\]
\end{proof}

\textbf{Example: }  $e(1)=6$ and $e(2)=38$.

Due to the direct connection $E[Y_A^r] = \dfrac{C[Y_A^r]}{2^n}$, we will present the procedure for the Middle and Low Levels in terms of $C[Y_A^r]$. 
\begin{tcolorbox}
\textbf{Middle Level:} Decompose $Y_A$ into a sum of indicator random variables:  
\[
 Y_A=Y_1^A+Y_2^A+\dots+Y_h^A,
\]
where $Y_i^A=1$ if the $i$th position is guessed correctly and comes from the first sequence, and $0$ otherwise.  \vspace{1em}

The $r$th moment of $Y_A$ can be found by
\begin{align}
\label{eq:C_A}
C[Y_A^r] &= C[(Y_1+Y_2+\dots+Y_h)^r]  \notag\\
&= \sum_{i_1=1}^h \sum_{i_2=1}^h \dots \sum_{i_r=1}^h C[Y_{i_1}Y_{i_2}\dots Y_{i_r}],
\end{align}
where for simplicity, we have suppressed the superscript``$A$'' on $Y_i$. We will continue to suppress such superscripts as their omission causes no confusion.
\end{tcolorbox}

As we can see, the method of overlapping stages in the Middle Level leads us to finding the formula for $C[Y_{i_1}Y_{i_2}\dots Y_{i_r}]$ in \eqref{eq:C_A}. Its formula when $1\leq i_1 < i_2 < \dots < i_r \leq h$  is given by \eqref{eq:low} below. 

\begin{tcolorbox}
\textbf{Low Level, the Building Block $C[Y_{i_1}Y_{i_2}\dots Y_{i_r}]$}  

\vspace{1em}
For $1\leq i_1 < i_2 < \dots < i_r \leq h$,
\begin{align} 
\label{eq:low}
&C[Y_{i_1}Y_{i_2}\dots Y_{i_r}] \notag\\
&=  2^n 
\binom{i_1-1}{\lf i_1/2\rf} \cdot \binom{i_2-i_1-1}{\lf i_2/2\rf- \lf i_1/2 \rf -1}
\dots \binom{i_r-i_{r-1}-1}{\lf i_r/2\rf- \lf i_{r-1}/2 \rf -1} \cdot \dfrac{1}{2^{i_r}}.
\end{align}
\end{tcolorbox}

\begin{rem} We only need the formula $C[Y_{i_1}Y_{i_2}\dots Y_{i_r}]$ for all distinct indices. If, for example, $i_1=i_2$, then $C[Y_{i_1}^2Y_{i_3}Y_{i_4}\dots Y_{i_r}]=C[Y_{i_1}Y_{i_3}Y_{i_4}\dots Y_{i_r}]$, and so the formula for $r-1$ distinct indices applies. 
\end{rem}

{\bf Justifying the formula for $C[Y_i]$}

When $r=1$,  $C[Y_i]$ corresponds to the number of ways the fixed position $i$ coming from the first sequence is guessed correctly.  The above formula is simplified to \eqref{eq:C_A_i}, for which we shall justify now using a combinatorial counting method.  

For $1 \leq i \leq h$, 
\begin{equation}
\label{eq:C_A_i}
 C[Y_i] =  \sum_{a=0}^h \sum_{b=0}^{n-h} 
\binom{i-1}{\lf i/2\rf}\cdot \binom{h-i}{a-\lf i/2\rf-1}\cdot \binom{n-h}{b}
= 2^n \binom{i-1}{\lf i/2\rf}\cdot \dfrac{1}{2^i}. 
\end{equation}

For a fixed $i$, $Y_i$ is the indicator variable indicating the correctness of the guess at the $i$th position and it must come from the first sequence. Figure \ref{fig:explaining_CY1} provides a visual explanation for the formula of $C[Y_i]$ in \eqref{eq:C_A_i}.

 \begin{figure}[h!]
    \centering
	\includegraphics[width=1\textwidth]{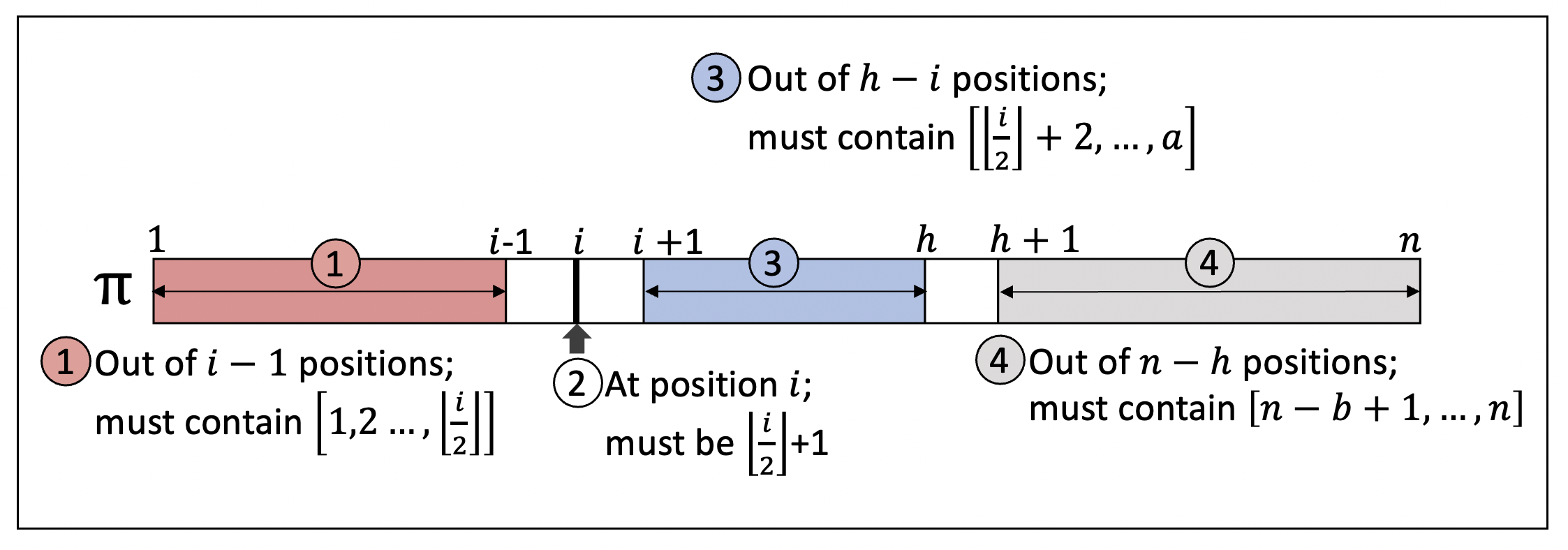} 
    \caption{Structure of $\pi$ for computing $C[Y_i]$ in \eqref{eq:C_A_i}}
    \label{fig:explaining_CY1}
\end{figure}

First, let us fix the values of $a$ (the length of $A$ in the top half) and $b$ (the length of $B$ in the bottom half). Then, in order to have the guess at position $i$ correctly, the permutation $\pi$ must be such that $\pi[i]=\lf i/2\rf+1$ (i.e. correct guess) and in addition this guess must come from $A$. This implies that top half of $\pi$ (of length $h$) can be divided into two blocks with the $i$th position as a separator. This is precisely \circled{1} and \circled{3} of  Figure \ref{fig:explaining_CY1}. As for the bottom half, there are $n-h$ positions left and these positions must contain the last $b$ elements of $B$ as illustrated by \circled{4} in the figure. The desired formula for $C[Y_i]$ in \eqref{eq:C_A_i} can be obtained by summing over all possible values of $a$ and $b$.

This idea can be generalized to justify the formula for $C[Y_{i_1}Y_{i_2}\dots Y_{i_r}]$, for $r\geq2$. For reference, we give a formula for $r=2$, which corresponds to the number of ways to obtain the correct guesses at two different positions $i$ and $j$ from the first sequence.

For $1 \leq i < j \leq h$,
\begin{align} C[Y_i Y_j]  &= \sum_{a=0}^h \sum_{b=0}^{n-h}
\label{eq:C_A_ij}
\binom{i-1}{\lf i/2\rf}\cdot \binom{j-i-1}{\lf j/2\rf-\lf i/2\rf-1}
\cdot \binom{h-j}{a-\lf j/2\rf-1}\cdot \binom{n-h}{b} \notag\\
&=  2^n \binom{i-1}{\lf i/2\rf}\cdot \binom{j-i-1}{\lf j/2\rf-\lf i/2\rf-1} \cdot \dfrac{1}{2^j}. 
\end{align}


\subsection{Closed-form expression for the first moment, $E[X]$\label{sec:closed-form-EX}}
Starting from the Low Level formula  $C[Y_i]$ and proceeding backwards, we provide detailed calculations for $\sum_{i=1}^h C[Y_i]$ of the Middle Level, and eventually for $E[X]$ of the High Level. 


\textbf{Middle Level:}
Equations \eqref{eq:C_A} and  \eqref{eq:C_A_i} for $r=1$ lead to
\[ C[Y_A] = \sum_{i=1}^h C[Y_i] = 2^n \sum_{i=1}^h \binom{i-1}{\lf i/2\rf} \dfrac{ 1 }{2^i},\]
which can be evaluated in closed form as we shall now demonstrate. In order to get rid of the floor function, we consider the following two cases.

Case 1:  $h$ is even.  Assume $h = 2L.$ Then,
\begin{align*} C[Y_A] &= 2^n \sum_{s=1}^{L}  \binom{2s-1}{s}\dfrac{1}{4^s}
+2^n \sum_{s=0}^{L-1}  \binom{2s}{s} \dfrac{1 }{2\cdot 4^s}
= 2^n \left[  \sum_{s=0}^{L-1} \binom{2s}{s} \dfrac{ 1}{4^s} 
+ \dfrac{1}{2}\binom{2L}{L}\dfrac{1}{4^L} -\dfrac{1}{2} \right]\\
&= 2^n \left[  \dfrac{2L+1/2}{4^L}\binom{2L}{L}-\dfrac{1}{2} \right].
\end{align*}

Case 2:  $h$ is odd.  Assume $h = 2L-1.$ Then,
\begin{align*}  C[Y_A] &= 2^n \sum_{s=1}^{L-1} \binom{2s-1}{s} \dfrac{ 1}{4^s}
+2^n \sum_{s=0}^{L-1} \binom{2s}{s} \dfrac{ 1 }{2\cdot 4^s}
= 2^n \left[ \sum_{s=0}^{L-1}  \binom{2s}{s}\dfrac{ 1}{4^s} -\dfrac{1}{2} \right]\\
&= 2^n \left[  \dfrac{2L}{4^L}\binom{2L}{L}-\dfrac{1}{2} \right].
\end{align*}

In both of these cases, we have used the identity 
\[  \sum_{s=0}^{L-1} \binom{2s}{s} \dfrac{ 1 }{ 4^s} = \dfrac{2L}{4^L}\binom{2L}{L}. \]

\textbf{High Level:} By the linearity 
\[  E[Y_A+Z_B] = E[Y_A]+E[Z_B],\]
and so 
\[ E[X] = E[Y_A+Z_B]+ \dfrac{6}{2^n},\]
where the term $ \dfrac{6}{2^n}$ is the excess term as derived in Proposition \ref{prop:excess}.

The formula of $E[X]$ therefore depends on the parity of $h$ and $n-h$, and we obtain the following theorem.
\begin{thm}
Let  $L = \lc h/2 \rc$ and $\alpha \in \{-1,0,1,2\}$. For $n=4L+\alpha,$
\begin{equation}
E[X] = (n+1-\alpha/2)\binom{2L}{L} \dfrac{1}{4^{L}} -1+\dfrac{6}{2^n}.
\end{equation}
\end{thm}

This theorem provides an exact, closed-form formula for $E[X]$. The leading term of $E[X]$ was derived earlier in \cite{C,KT2}: \[  E[X] = \dfrac{2\sqrt{n}}{\sqrt{\pi}}+\mathcal{O}(1).\] 


\subsection{Closed-form expression for the second moment, $E[X^2]$}

Recalling the Low Level formula \eqref{eq:C_A_ij} 
\[ C[Y_i Y_j] = 
2^n \binom{i-1}{\lf i/2\rf}\cdot \binom{j-i-1}{\lf j/2\rf-\lf i/2\rf-1} \cdot \dfrac{1}{2^j},
\]
for $1 \leq i < j \leq h$, we now give detailed calculations for $E[X^2]$ for the case when $h$ is even.

\textbf{Middle Level:}

To obtain $C[Y^2_A]$, we evaluate the summation
$ \sum_{i=1}^h \sum_{j=i+1}^h C[Y_iY_j] $ in \eqref{eq:C_A}
to the closed form after getting rid of the floor function.
We start from the inner sum.

Case 1: $h=2L$ and $i=2s$
\begin{align*} 
\sum_{j=i+1}^h  \binom{j-i-1}{\lf j/2\rf- \lf i/2 \rf -1} \dfrac{1}{2^j}
&= \dfrac{(2L-2s+3/2)}{4^L}\binom{2L-2s}{L-s}-\dfrac{3}{2\cdot 4^s} .
\end{align*}

Case 2: $h=2L$ and $i=2s-1$
\[
\sum_{j=i+1}^h \binom{j-i-1}{\lf j/2\rf- \lf i/2 \rf -1} \dfrac{1}{2^j}
= \dfrac{(4L-4s+2)}{4^L}\binom{2L-2s}{L-s}-\dfrac{1}{4^s}.\]

Hence, for $h=2L,$
\begin{align}
\dfrac{1}{2^n}\sum_{i=1}^h \sum_{j=i+1}^h C[Y_iY_j] &= 
\sum_{i=1}^h \binom{i-1}{\lf i/2\rf} \cdot \sum_{j=i+1}^h \binom{j-i-1}{\lf j/2\rf- \lf i/2 \rf -1} \dfrac{1}{2^j} \notag\\
&= (L+3/2)-\binom{2L}{L}\dfrac{(3L+3/2)}{4^L}. 
\label{eq:formula_r2}
\end{align}

Thus, we obtain the second moment of $Y$ for $h=2L$:
\begin{align*}
C[Y^2_A] &= C[Y_A]+2\sum_{i=1}^h \sum_{j=i+1}^h C[Y_iY_j]  \\
&=  \dfrac{2^n}{4^L}(2L+1/2)\binom{2L}{L}-\dfrac{2^n}{2}
+2^{n+1}\left((L+3/2)-\binom{2L}{L}\dfrac{(3L+3/2)}{4^L}\right) \\
&= 2^n\left[ (2L+\dfrac{5}{2})- \dfrac{1}{4^L}\binom{2L}{L}(4L+5/2)\right].
\end{align*}

\textbf{High Level: Wrap it up!}
In this final step, we show how to obtain $E[X^2]$ for the case $n =4L$. 

\begin{align*} 
E[(Y_A+Z_B)^2] &= E[Y^2_A]+E[Z_B^2]+2 E[Y_A]E[Z_B] \\
&=   4L+5- \dfrac{1}{4^L}\binom{2L}{L}(8L+5)
+\dfrac{2}{16^L} \left((2L+1/2)\binom{2L}{L}-\dfrac{4^L}{2}\right)^2.
\end{align*}
Then,
\[E[X^2] =E[ (Y_A+Z_B)^2]+\dfrac{38}{2^n}, \]
where  $38$ is the excess constant term $e(2)$ given in Proposition \ref{prop:excess}.

For other cases as well as higher moment $E[X^r]$, to avoid tedious computation, we shall adopt the interpolation method to recover a closed-form formula in the next section.

\subsection{The higher moment $E[X^r]$ and the main theorem\label{sec:highermoments}}

The reader may have noticed that the bottleneck of the procedure is the derivation of a closed-form expression for $\sum C[Y_{i_1}Y_{i_2}\dots Y_{i_r}]$ in the Middle Level, which is extremely tedious even for the second moment as we have to prove the formula using the case analysis due to the floor function.
The simplicity of the formulas obtained for the first and second moments, however, suggests us that there might be a certain form for higher moments.
We state this formally in Theorem \ref{Main}.

\begin{thm} \label{Main}
Let  $L = \lc h/2 \rc$. Then, 
\begin{equation} 
\label{eq:sum_CYr}
\sum_{1\leq i_1 < i_2 < \dots < i_r \leq h}C[Y_{i_1}Y_{i_2}\dots Y_{i_r}]
= 2^n \left[ P(L)\binom{2L}{L}\cdot\dfrac{1}{4^L}+Q(L) \right],  
\end{equation} 
for some polynomials $P(L)$ and $Q(L)$ with degrees no more than $\lc r/2 \rc$
and $\lf r/2 \rf$, respectively.
\end{thm}

\begin{rem}
Let us note that for each $r$, the expression on the right hand side of \eqref{eq:sum_CYr} depends on the parity of $h$. For example, for $r=2$ and $h=2L$, the previous result in \eqref{eq:formula_r2} implies that $P(L)=-(3L+3/2)$ and $Q(L)=L+3/2$.
\end{rem}

The proof of the theorem is quite technical and eight pages long. Despite having its own merits, we defer it to Appendix so as not to interrupt the flow of the~paper.

To shorten the notation, 
we denote $M_r(h) := \sum C[Y_{i_1}Y_{i_2}\dots Y_{i_r}]$.

Recall that $Y_A = Y_1+Y_2+\dots+Y_h$, and so the expression of $C[Y^r_A]$ can be written in terms of $M_r(h)$ as follows:
\begin{align*}
C[Y_A] &=   M_1(h), \\
C[Y^2_A] &=   M_1(h)+2M_2(h), \\
C[Y^3_A] &=   M_1(h)+6M_2(h)+6M_3(h). 
\end{align*}
In general, we obtain the following proposition by a counting method.

\begin{prop}
Let $M_r(h) = \sum_{1\leq i_1 < i_2 < \dots < i_r \leq h} C[Y_{i_1}Y_{i_2}\dots Y_{i_r}]$ for $r\geq1$. Then,
\begin{equation}  
\label{eq:integer_partition}
C[Y^r_A] =  \sum_{P \in Par(r)} 
\dfrac{r!}{(p_1!)^{m_1}(p_2!)^{m_2}\dots (p_s!)^{m_s}}
\dfrac{m!}{m_1!m_2!\dots m_s!} M_m(h), 
\end{equation}  
where each partition $P$ of size $r$ is such that $P= p_1^{(m_1)}+p_2^{(m_2)}+\dots+p_s^{(m_s)}.$ With this notation, $m_i$ is the number of multiples of $p_i$ in partition $P$. 
Hence, $m$ is the number of parts of partition $P$, i.e. $m=m_1+m_2+\dots+m_s$.
\end{prop}

The formula \eqref{eq:integer_partition} may seem a little complicated at first. However, the fact that $C[Y^r_A]$, as a linear combination of $M_i,\, 1 \leq i \leq r,$ remains in the same form leads us to the following corollary. Remarkably, this fact  allows us to avoid the explicit calculation of \eqref{eq:integer_partition}.

\begin{cor} 
\label{cor:linear combination}
Let  $L = \lc h/2 \rc$. Then,
\begin{equation}
\label{eq:linear combination}
C[Y^r_A]=2^n \left[ P(L)\binom{2L}{L}\cdot\dfrac{1}{4^L}+Q(L) \right],  
\end{equation}
for some polynomials $P(L)$ and $Q(L)$ with degrees no more than $\lc r/2 \rc$ and $\lf r/2 \rf$, respectively.
\end{cor}

Corollary \ref{cor:linear combination}, together with the numeric expression of $C[Y^r_A]$ we developed earlier in Section \ref{sec:numericMoment}, is the key to bypass completely the step of deriving a closed-form formula for $\sum C[Y_{i_1}Y_{i_2}\dots Y_{i_r}]$ for $r\geq0$, and instead directly derive a closed-form formula for $C[Y^r_A]$ as we shall now see.

\textbf{An unexpected (but life-saving) application of interpolation to find a closed-form formula for $C[Y^r_A]$}

A numeric expression of $C[Y^r_A]$ from generating functions can be used to find its closed-form formula via interpolation. To be more precise, from~\eqref{eq:linear combination}, 
\[
C[Y^r_A]=2^{n-h} \left\{4^L \left[ P(L)\binom{2L}{L}\cdot\dfrac{1}{4^L}+Q(L) \right]\right\}=2^{n-h} \left[ P(L)\binom{2L}{L}+Q(L)\cdot 4^L  \right].
\]
On the other hand, for fixed values of $r$ and $h$, we can compute the value of $D^{(r)}F_A(q;h)\vert_{q=1}$ where $F_A(q;h)=\sum_{a=0}^h G(a,h-a,q)$, and use Corollary \ref{cor:CY_A,CZ_B} to connect it to $C[Y^r_A](h)$ through the use of numeric expression:
\begin{equation}
\label{eq:F_A}
C[Y^r_A](h)= 2^{n-h} D^{(r)}F_A(q;h)\vert_{q=1}.
\end{equation}

We thus can make use of the numeric expression to generate a dataset required for fitting $P(L)$ and $Q(L)$. Since the polynomial expressions of $P(L)$ and $Q(L)$ depend on the parity of $h$,  we explain briefly how to carry out interpolation for the even~$h$.

\begin{algorithm}
\caption{{\bf Interpolation Procedure for $h=2L$ ($r$ fixed)}}
    {\bf Step 1:} Generate enough data $\left\{D^{(r)}F_A(q;2L)\vert_{q=1}: L=1,2,3,\dots...\right\}$. (At least $\lc r/2 \rc+\lf r/2 \rf+2$ data points are required to fit polynomials  $P(L)$ and $Q(L)$ with the desired degrees.)\\
   {\bf Step 2:} Define the function $f(L)=P(L)\binom{2L}{L}+Q(L)\cdot 4^L $  and use the data generated in Step 1 to fit $f(L)$ and recover the polynomials $P(L)$ and $Q(L)$.\\
   {\bf Step 3:}  Use \eqref{eq:linear combination} to obtain the closed-form expression for $C[Y^r_A](h)$ for the even $h$.
\end{algorithm}

Note that with the fastest version \circled{III} of the generating function, the dataset in Step 1 can be generated very efficiently. Following the procedure, the closed-form expression for $C[Y^r_A]$ can be obtained for even $h$. The formula when $h$ is odd can be obtained in the same manner. 

Given $h= \lc n/2 \rc$ and
$L = \lc h/2\rc$, the list of $C[Y^r_A]$ for $r=1, 2, 3$ is given below.

For even $h$:
\begin{align*} C[Y_A] &= 2^{n-h}\left[ (2L+1/2)\binom{2L}{L} -\dfrac{1}{2} \cdot 4^L \right], \\
C[Y_A^2] &= 2^{n-h}\left[ -(4L+5/2)\binom{2L}{L} +(2L+5/2) \cdot 4^L \right], \\
C[Y_A^3] &= 2^{n-h}\left[ (8L^2+24L+19/2)\binom{2L}{L} -(9L+19/2) \cdot 4^L \right]. 
\end{align*}

For odd $h$:
\begin{align*} C[Y_A] &= 2^{n-h}\left[ L \cdot\binom{2L}{L} -\dfrac{1}{4} \cdot 4^L \right], \\
C[Y_A^2] &= 2^{n-h}\left[ -(2L+1)\binom{2L}{L} +(L+3/4) \cdot 4^L \right], \\
C[Y_A^3] &= 2^{n-h}\left[ (4L^2+9L+3)\binom{2L}{L} -(9L/2+13/4) \cdot 4^L \right]. 
\end{align*}

{\bf$\mathtt</\hspace{-0.3em}>$}Maple command to find the polynomials inside the bracket is {\bf\texttt{FitMoFA(h1,r,n,L)}}.\\
For example, for the odd $h$ of $C[Y_A^3]$, try
{\bf\texttt{FitMoFA(1,3,n,L)}}.

In addition, the list of the moments for the case $n=4L$ is given as follows.
\begin{align*}
E[X^0] &= 1\\
E[X\,] &= \dfrac{(4L+1)}{4^L}\binom{2L}{L}-1+\dfrac{6}{2^{4L}} \\
E[X^2] &= \dfrac{(4L+1)^2}{2\cdot4^{2L}}\binom{2L}{L}^2
-\dfrac{6(2L+1)}{4^L}\binom{2L}{L}+4L+\dfrac{11}{2}+\dfrac{38}{2^{4L}} \\
E[X^3] &= -\dfrac{3(8L+5)(4L+1)}{2\cdot4^{2L}}\binom{2L}{L}^2
+\dfrac{2(20L^2+48L+17)}{4^L}\binom{2L}{L}-24L-\dfrac{53}{2}+\dfrac{186}{2^{4L}} \\
E[X^4] &= \dfrac{(256L^3+1024L^2+736L+151)}{2\cdot4^{2L}}\binom{2L}{L}^2
-\dfrac{8(50L^2+94L+33)}{4^L}\binom{2L}{L}\\
&+48L^2+232L+\dfrac{377}{2}+\dfrac{830}{2^{4L}} \\
E[X^5] &= -\dfrac{15(256L^3+736L^2+508L+101)}{2\cdot4^{2L}}\binom{2L}{L}^2
+\dfrac{2(344L^3+2558L^2+3610L+1163)}{4^L}\binom{2L}{L}\\
&-720L^2-2280L-\dfrac{3137}{2}+\dfrac{3546}{2^{4L}}.
\end{align*}

{\bf$\mathtt</\hspace{-0.3em}>$}Maple command to find these moments is {\bf\texttt{ForMoX(n1,r,L)}}.\\
For example, for the result of $E[X^5]$ above, try
{\bf\texttt{ForMoX(0,5,L)}}.

\subsection{Distribution of $X$ for a one-time shuffled deck}

To complete the section for a 1-time shuffle, we provide the list of the moments about the mean for $n=4L$.
\begin{align*}
E[X-\mu] &= 0,\\ 
E[(X-\mu)^2] &= -\dfrac{(4L+1)^2}{2\cdot4^{2L}}\binom{2L}{L}^2
-\dfrac{4(L+1)}{4^L}\binom{2L}{L}+4L+\dfrac{9}{2}+o(1),\\
E[(X-\mu)^3] &= \dfrac{(4L+1)^3}{2\cdot4^{3L}}\binom{2L}{L}^3
+\dfrac{6(L+1)(4L+1)}{4^{2L}}\binom{2L}{L}^2 \\
&-\dfrac{(8L^2-6L-11/2)}{4^L}\binom{2L}{L}-12(L+1) +o(1), \\
E[(X-\mu)^4] &= 
-\dfrac{128L^3+320L^2+128L+1/2}{4^{2L}}\binom{2L}{L}^2 \\
&-\dfrac{48L^2+184L+112}{4^L}\binom{2L}{L}+48L^2+160L+225/2
+o(1). 
\end{align*}

{\bf$\mathtt</\hspace{-0.3em}>$}Maple command to find these moments is {\bf\texttt{ForMoXMean(n1,r,L)}}.\\
For example, for the result of $E[(X-\mu)^4]$ above, try
{\bf\texttt{ForMoXMean(0,4,L)}}.

Finally, the sequence (in $r$) of the $r$th standardized moments as $L$ approaches infinity, i.e.
$\displaystyle \lim_{L \to \infty}\dfrac{E[(X-\mu)^r]}{E[(X-\mu)^2] ^{r/2}},$ 
is found to be
\[ \left[ 0, 1, \dfrac{4-\pi}{(\pi-2)^{3/2}} , \dfrac{(3\pi-8)\pi}{(\pi-2)^2}, \dots  \right] .\]

This is a little disappointment as the sequence does not seem to fit any existing distributions (as far as we know).
However, it confirms the non-normality of the distribution we observed earlier in the histograms of Figure \ref{fig:histogram}.

\section{$k$-shuffles, $k \geq 2$ \label{sec:k_shuffles}}

Riffle-shuffle a deck of $n$ cards for $k\geq2$ times, where initially the cards are ordered from $1$ to $n$. Assume that the player guesses the cards with the optimal guessing strategy ${\cal{G}}^*$ for all $n$:
\[
\underbrace{1,\dots, 1}_{2^k-1\text{ times}}, \;\underbrace{2,\dots, 2}_{2^k\text{ times}}, \;\underbrace{3,\dots, 3}_{2^k\text{ times}}, \;\underbrace{4,\dots, 4}_{2^k\text{ times}}, \dots  
\]
for the top half and the bottom half in the reverse manner, i.e. 
\[
\dots, \underbrace{n-3,\dots, n-3}_{2^k\text{ times}}, \;\underbrace{n-2,\dots, n-2}_{2^k\text{ times}},\; \underbrace{n-1,\dots, n-1}_{2^k\text{ times}}, \;\underbrace{n,\dots, n}_{2^k-1\text{ times}}. 
\]

For $k=1$ shuffle, we have seen that the size of the sample space is $2^n$, where each permutation has two (increasing) sequences $A$, $B$ according to Definition \ref{def:increasing_seq}. Similarly, for $k>1$ shuffles, the size of the sample space is $2^{kn}$, where each permutation has $2^k$ (increasing) sequences. 

\subsection{First moment}  
In this section, we use $C$ to denote the number of (increasing) sequences in the permutation. For example, if the permutation is $[4,3,1,2]$ and the sequences are $[4],[3],[1,2]$, then $C=3$. The number of possible  outcomes of permutation with $C$ increasing sequences is $C^n$ (easily proved by a bijective proof).  But of course, this permutation can never be obtained by riffle shuffles because a $k$-time riffle-shuffled deck will always give $2^k$ sequences.  However, since the identity \eqref{eq:k_shuffles} holds for any positive integer $C\geq1$, we state and prove the theorem for a permutation with any number of increasing sequences in general. 
Note also that the optimal guessing strategy ${\cal{G}}^*$ is equivalent to guessing the number $\lf i/C \rf+1$ for the card position $i$.
The subscript $C$ on $E_C[\,\cdot\,]$ is simply used to remind us of the number of sequences.

\begin{thm}
\label{thm:k_shuffles}
Let $C$ be a positive integer representing the number of sequences allowed for the permutation.  Let $X$ be a random variable representing the number of correct guesses under ${\cal{G}}^*$. Decompose $X=Y+Z$, where $Y$ and $Z$ are the number of correct guesses in the top and bottom halves of the deck. Let $h=\lc \dfrac{n}{2} \rc$ be the length of the top half. Then,
\[ E_C[X] = E_C[Y]+E_C[Z],\]
where
\begin{small}
\begin{equation} 
\label{eq:k_shuffles}
C^n \cdot E_C[Y] = \sum_{i=1}^h \sum_{m=1}^C \sum_{T+S=\lf i/C \rf } 
\left[ \binom{i-1}{T} m^T(C-m)^{i-1-T}\right] \cdot \left[\binom{n-i}{S} (m-1)^S(C-m+1)^{n-i-S}\right]. 
\end{equation}
\end{small}
\end{thm}

\begin{proof} We give a combinatorial proof to \eqref{eq:k_shuffles}, which determines the sum of the number of correct guesses in the top half of all possible permutations with $C$ increasing sequences, $C_C[Y]$. 
Starting from the outermost sum, a fixed index $i$ denotes the position in the top half ($1\leq i\leq h$) with a correct guess under ${\cal{G}}^*$, i.e. the value of the card at position $i$ is $\lf i/C \rf+1$. 
For a fixed position $i$, this value can come from any sequence $m$th. Hence, the middle sum sums over all these possibilities ($1\leq m \leq C$).

For $1 \leq t \leq C$, let $A_t$ ($B_t$) be the portion of the $t$-th sequence that comes before (after) $i$. Also, we denote by $a_t$ ($b_t$)  the length of $A_t$ ($B_t$).

For a fixed position $i$, we have the condition
\[  \sum_{t=1}^{C} a_t = i-1.  \]

Moreover, if the value $v$ of the $i$th position comes from the $m$th sequence, $1\leq m \leq C$, then the following identity holds:
\[ \sum_{t=1}^m a_t + \sum_{t=1}^{m-1} b_t = v-1.   \]
 \begin{figure}[h!]
    \centering
	\includegraphics[width=0.8\textwidth]{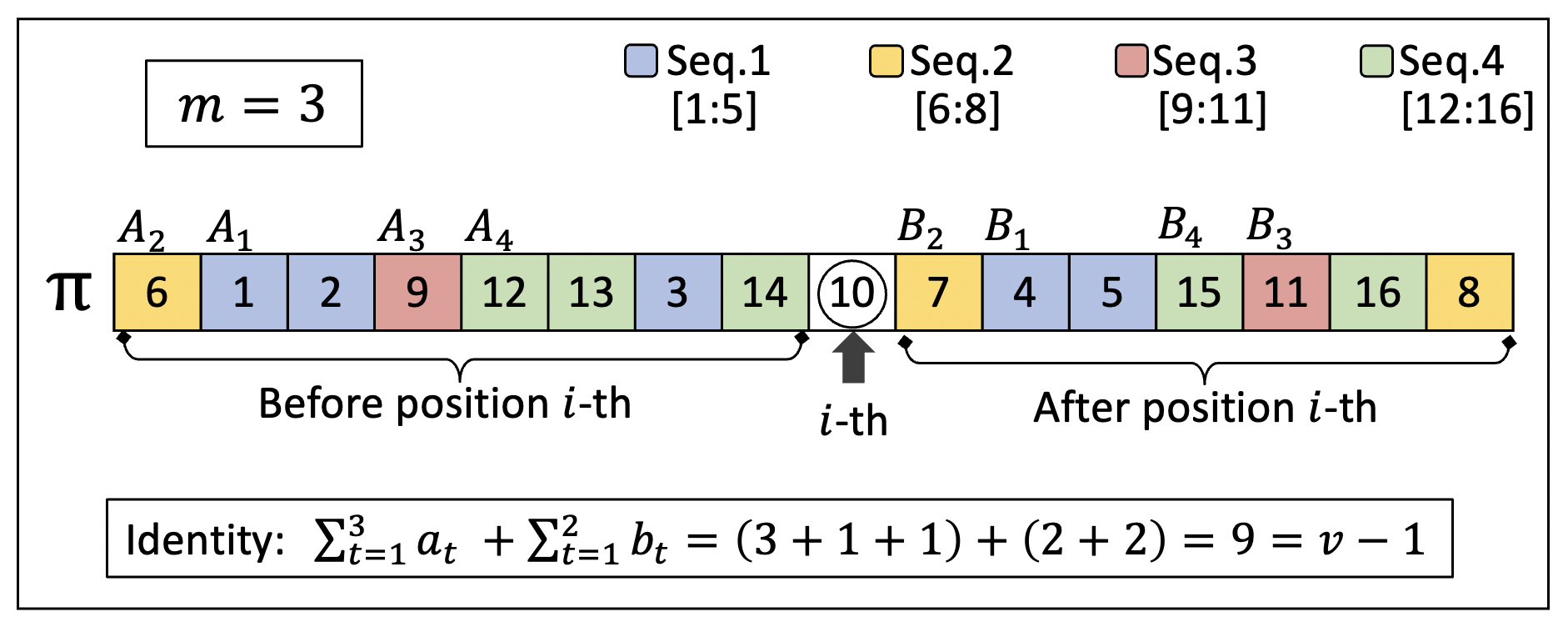} 
    \caption{Example of $\sum_{t=1}^m a_t + \sum_{t=1}^{m-1}b_t =v-1$. Here, $n=16$, $C=4$, $m=3$, $v=10$.}
    \label{fig:explaning_k_shuffles}
\end{figure}

To see why this statement holds, let us consider an example given in Figure \ref{fig:explaning_k_shuffles} where there are total of 4 increasing sequences. The $i$th position  splits each of these sequences into the portions ``before'' and ``after'', namely, $A_1=[1:3]$, $A_2=[6]$, $A_3=[9]$, $A_4=[12:14]$ and $B_1=[4:5]$, $B_2=[7:8]$, $B_3=[11]$, $B_4=[15:16]$. Since the $i$th position comes from the 3rd sequence ($m=3$) with the corresponding card value $v=10$, we can see that 
\[ \sum_{t=1}^m a_t + \sum_{t=1}^{m-1} b_t = (3+1+1)+(2+2)=9=v-1. \]

In particular, if we apply the statement above to the correct position, the statement will give the necessary and sufficient condition for the correct position. By letting $T$ and $S$ be the two sums $\sum_{t=1}^m a_t$ and $\sum_{t=1}^{m-1} b_t$, respectively, and noting the value at the correct position $i$ being $v=\lf i/C \rf+1$, we obtain the bound of the innermost sum $T+S=v-1=\lf i/C \rf $.

Finally, the first (second) bracket of \eqref{eq:k_shuffles} represents the number of ways to place cards before (after) position $i$. The details of the first bracket are as follows. The term $\binom{i-1}{T}$ comes from the fact that there are $i-1$ positions, $T$ of which is needed to place the sequence $A_1, \dots, A_m$ in. Once the $T$ positions have been fixed, there are $m$ possibilities (as there are $m$ sequences) from which each position can come, and $C-m$ possibilities for each of the other $i-1-T$ positions, leading to the factor $m^T(C-m)^{i-1-T}$. The second bracket can be justified in the same manner.
\end{proof}

{\bf Combinatorial interpretation and the connection with results of \cite{KT2}}

It is very likely that there is no closed-form formula for $E_C[X]$ 
except for the case $C=2$ for which the expression is reduced to the case $k=1$ shuffle. However, for fixed values of position $i$ and sequence $m$, the formula $\sum_{T+S=\lf i/C \rf}\left[ \binom{i-1}{T} \left(\dfrac{m}{C}\right)^T(1-\dfrac{m}{C})^{i-1-T}\right] \cdot \left[\binom{n-i}{S} \left(\dfrac{m-1}{C}\right)^S(1-\dfrac{m-1}{C})^{n-i-S}\right]$ in \eqref{eq:k_shuffles}
is a Poisson binomial distribution, i.e. flipping two coins for a total of  $n-1$ times, where the first $i-1$ flips come from the first coin with probability $p=\dfrac{m}{C}$ of getting a head, and the other $n-i$ flips come from the second coin with probability $q=\dfrac{m-1}{C}$ of getting a head. Then, this expression is the probability $P\left(N=\lf i/C \rf\right)$ where $N$ is a random variable counting the number of heads out of these $n-1$ flips.

We were surprised at first to find the resemblance between this formula and the generating function in \cite[Theorem 2]{KT2} as we established the result here with a completely different approach. However, it turns out that this is not at all a coincidence. While the proof can also be derived algebraically using the $k$-step probability transition matrix given in  \cite{KT2}, with the method of overlapping stages, we were able to avoid the algebraic proof, and instead provided a self-contained combinatorial proof.
Nevertheless, by recalling several objects introduced in  \cite{KT2}, the interpretation we just gave is precisely the combinatorial interpretation for the generating function object $F_L$ appearing in \cite[Theorem 2]{KT2}. This is the answer to the open problem 1 of \cite{KT2}. 

In addition, the result of \cite[Section 4]{KT2} indicates that as $n \to \infty,$ the dominated term of  \eqref{eq:k_shuffles} comes from $m=1$ only. Thus, keeping only the leading term $m=1$,  \eqref{eq:k_shuffles} is simplified to
\[ E_C[Y] = \dfrac{1}{C^n} \sum_{i=1}^h 
\binom{i-1}{\lf i/C \rf}(C-1)^{i-1-\lf i/C \rf}C^{n-i}+\mathcal{O}(1) = \sqrt{\dfrac{n}{(C-1)\pi}}+\mathcal{O}(1), \]
and hence
\[   E_C[X] = 2\sqrt{\dfrac{n}{(C-1)\pi}}+\mathcal{O}(1).  \]
In particular, when $C=2^k$, we obtain the leading term for $E_{2^k}[X]$, which was found previously in \cite{KT2} using a totally different proof approach. 

\subsection{Comments on higher moments} 

Calculation of higher moments for $k$ shuffles is even more challenging.
The generating function method (even with the method \circled{II}) becomes very slow as the number of variable grows exponentially with the number of shuffles $k$. Because all the $2^k$ sequences greatly contribute to the number of correct guesses, it is very likely that a fastest calculation method \circled{III} does not exist anymore. 

As for the method of overlapping stages (either for a closed-form formula or a nested-sum),
we have to figure out the calculation for $E[X_{[1]}X_{[2]}]$ (to begin with), where $X_{[i]}$ is a random variable that counts the number of correct guesses coming from sequence $i$th.
The non-independence between $X_{[1]}$ and $X_{[2]}$ makes the application of the method of overlapping stages on the higher moments very difficult. 
Although obtaining the leading term approximation (that likely comes from the sequence $m=1$ only) seems doable, we leave this for future work.

\subsubsection*{Conflict of interest statement} \vspace{-1em}
On behalf of all authors, the corresponding author states that there is no conflict of interest.


\appendix
\newpage
\section{Proof of Theorem \ref{Main}\label{sec:appA}}

We will prove the following claim: 
\[
\dfrac{1}{2^n}\sum_{1\leq i_1 < i_2 < \dots < i_r \leq h}C[Y_{i_1}Y_{i_2}\dots Y_{i_r}]
= P(L)\binom{2L}{L}\cdot\dfrac{1}{4^L}+Q(L),
\]
for some polynomials $P(L)$ and $Q(L)$ with degrees no more than $\lc r/2 \rc$
and $\lf r/2 \rf$, respectively.

In what follows, we will establish the results for the case $h=2L$ using generating functions to evaluate combinatorial sums in order to reveal the true structure of the object. For the case $h=2L-1$, the proof can be  carried out in very much the same way.

Recalling the expression of $C[Y_{i_1}Y_{i_2}\dots Y_{i_r}]$ in \eqref{eq:low}, 
we denote by $F^{(r)}(y)$ the generating function of the left-hand-side expression above (to be referred to as the ``original sum'' hereafter). Then,
\begin{small}
\begin{align*}
F^{(r)}(y) &:=\sum_{L=0}^{\infty}\left(\dfrac{1}{2^n}\sum_{1\leq i_1 < i_2 < \dots < i_r \leq h}C[Y_{i_1}Y_{i_2}\dots Y_{i_r}]\right)y^L\\
&=\sum_{L=0}^{\infty}\left(\sum_{1\leq i_1 < i_2 < \dots < i_r \leq 2L}  
\binom{i_1-1}{\lf i_1/2\rf}  \binom{i_2-i_1-1}{\lf i_2/2\rf- \lf i_1/2 \rf -1} \dots
\binom{i_r-i_{r-1}-1}{\lf i_r/2\rf- \lf i_{r-1}/2 \rf -1}  \dfrac{1}{2^{i_r}}  \right)y^L.     
\end{align*}
\end{small}

To demonstrate how we can handle the nested sum, let us start with $r=1$. Depending on the parity of the index $i$, the original sum can be decomposed into 
\begin{align*}
\dfrac{1}{2^n}\sum_{i=1}^{2L} C[Y_{i}] =  \sum_{i=1}^{2L} \binom{i-1}{\lf i/2\rf} \dfrac{ 1 }{2^i}
=  \underbrace{\sum_{s=0}^{L-1} \binom{2s}{s}\cdot \dfrac{1}{2}\cdot \dfrac{1}{4^s}}_{[0]}+\underbrace{\sum_{s=0}^{L-1} \binom{2s+1}{s}\cdot \dfrac{1}{4} \cdot \dfrac{1}{4^s}}_{[1]}.
\end{align*}
Recall that we have already derived the closed form of this expression in Section \ref{sec:closed-form-EX}. 
Here, we will solve for its formula through $F^{(1)}(y)$. To this end, we apply the generating function and evaluate each of the sums separately.
\begin{align*} 
F_{[0]}(y) &= \sum_{L=0}^{\infty}\sum_{s=0}^{L-1}
\binom{2s}{s} \dfrac{1}{2} \cdot \dfrac{1}{4^s} \cdot y^L 
= \dfrac{1}{2} \cdot \sum_{s=0}^{\infty}
\binom{2s}{s}  \dfrac{1}{4^s} \cdot \sum_{L=s+1}^{\infty}y^L \\
&=  \dfrac{y}{2(1-y)} \cdot \sum_{s=0}^{\infty}
\binom{2s}{s}  \left(\dfrac{y}{4}\right)^s 
=\dfrac{y}{2(1-y)^{3/2}}. 
\end{align*}
\begin{align*} 
F_{[1]}(y) &= \sum_{L=0}^{\infty}\sum_{s=0}^{L-1}
\binom{2s+1}{s} \dfrac{1}{4} \cdot \dfrac{1}{4^s} \cdot y^L 
= \dfrac{1}{4} \cdot \sum_{s=0}^{\infty}
\binom{2s+1}{s}  \dfrac{1}{4^s} \cdot \sum_{L=s+1}^{\infty}y^L \\
&= \dfrac{y}{4(1-y)} \cdot \sum_{s=0}^{\infty}
\binom{2s+1}{s}  \left(\dfrac{y}{4}\right)^s  \\
&= \dfrac{1-\sqrt{1-y}}{2(1-y)^{3/2}} 
= \dfrac{1}{2(1-y)^{3/2}}-\dfrac{1}{2(1-y)} . 
\end{align*}
Hence, for $r=1$,
\[ F^{(1)}(y) = F_{[0]}(y)+F_{[1]}(y) 
= \dfrac{y+1}{2(1-y)^{3/2}}-\dfrac{1}{2(1-y)} .\]
Let us note that the closed-form formula of the original sum is the coefficient of $y^L$ in the final expression (after expanding it into a power series). We defer the discussion of this to Proposition \ref{prop:expansion}.

For the case $r=2$, depending on the parity of the indices $i, j$, the original sum can be decomposed into four cases
\begin{align*}
& \dfrac{1}{2^n}\sum_{1\leq i<j \leq 2L} C[Y_{i}Y_{j}] \\
&=  \sum_{1\leq i<j \leq 2L}
\binom{i-1}{\lf i/2\rf} \binom{j-i-1}{\lf j/2\rf- \lf i/2 \rf -1} \cdot \dfrac{1}{2^{j}}  \\
&= \underbrace{\dfrac{1}{4} \sum_{s=0}^{L-1} \sum_{t=s}^{L-1} \binom{2s}{s}
\binom{2t-2s}{t-s}\cdot \dfrac{1}{4^t}}_{[0,0]}
+\underbrace{\dfrac{1}{8} \sum_{s=0}^{L-1} \sum_{t=s}^{L-2} \binom{2s}{s}
\binom{2t-2s+1}{t-s}\cdot \dfrac{1}{4^t}}_{[0,1]}\\
&+\underbrace{\dfrac{1}{8} \sum_{s=0}^{L-1} \sum_{t=s}^{L-2} \binom{2s+1}{s}
\binom{2t-2s}{t-s-1}\cdot \dfrac{1}{4^t}}_{[1,-1]}
+\underbrace{\dfrac{1}{16} \sum_{s=0}^{L-1} \sum_{t=s}^{L-2} \binom{2s+1}{s}
\binom{2t-2s+1}{t-s}\cdot \dfrac{1}{4^t}}_{[1,1]}.
\end{align*}

To obtain $F^{(2)}(y)$, we again apply the generating function to each of the sums separately. We leave it to the interested reader to verify the following identities, which will be used repeatedly when manipulating the nested sums for the generating functions. 

\begin{prop}
\label{prop:three_identities}
\begin{align*}
\sum_{s=0}^{\infty} \binom{2s}{s} \left(\dfrac{y}{4}\right)^s &= \dfrac{1}{\sqrt{1-y}}, \tag{A[0]} \\
\sum_{s=0}^{\infty} \binom{2s+1}{s} \left(\dfrac{y}{4}\right)^s 
&= \dfrac{2(1-\sqrt{1-y})}{y\sqrt{1-y}}, \tag{A[1]}\\
\sum_{s=0}^{\infty} \binom{2s}{s-1} \left(\dfrac{y}{4}\right)^s 
&= \dfrac{(1-\sqrt{1-y})^2}{y\sqrt{1-y}}. \tag{A[-1]}
\end{align*}
\end{prop}

\begin{notation}
There is a specific meaning to the labels $[0,0]$, $[0,1]$, $[1,-1]$, and $[1,1]$ of each part of the decomposed sum. In fact, it specifies the order in which the summation identities $(A[0])$, $(A[1])$, or $(A[-1])$ are applied when evaluating the nested sum. The label is, however, read from right to left, as the nested sum is evaluated from innermost (right) to outermost (left) one.
\end{notation}

Let us take a look at the generating function of the first case whose label is $[0,0]$. After interchanging the order of summation (pushing the summation index $L$ to innermost and evaluating it), the identity $(A[0])$ is applied twice to deal with the summation with respect to the indices $t$ and $s$, respectively, and we obtain:
\begin{align*} 
F_{[0,0]}(y) &= \dfrac{1}{4} \sum_{L=0}^{\infty}\sum_{s=0}^{L-1}\sum_{t=s}^{L-1}
\binom{2s}{s} \binom{2t-2s}{t-s}  \dfrac{1}{4^t} \cdot y^L 
= \dfrac{1}{4} \sum_{s=0}^{\infty}\sum_{t=s}^{\infty} 
\binom{2s}{s} \binom{2t-2s}{t-s} \dfrac{1}{4^t} \cdot \sum_{L=t+1}^{\infty}y^L \\
&= \dfrac{y}{4(1-y)} \sum_{s=0}^{\infty}
\binom{2s}{s}\cdot \sum_{t=s}^{\infty} \binom{2t-2s}{t-s} \left( \dfrac{y}{4} \right)^t 
\stackrel{\text{(A[0])}}{=} \dfrac{y}{4(1-y)^{3/2}} \cdot \sum_{s=0}^{\infty}
\binom{2s}{s} \left( \dfrac{y}{4} \right)^s\\ 
&\stackrel{\text{(A[0])}}{=} \dfrac{y}{4(1-y)^2}.
\end{align*}

We note that with this notation, the second value of each label (corresponding to the innermost summation index $t$) being $[0],[1],[-1]$ can be used to traced back the parity of the indices $i$ and $j$ in the original decomposed sum. For example, $[1]$ corresponds to the case whose indices $i$ and $j$ have same parity. $[0]$ shows up when $i$ is odd and $j$ is even. Finally, $[-1]$ shows up when $i$ is even and $j$ is odd.

Following the same procedure and applying the identities in an appropriate order, 
we obtain the expressions of the other three cases:
\begin{align*} 
F_{[0,1]}(y) &= \dfrac{1}{8} \sum_{L=0}^{\infty}\sum_{s=0}^{L-1}\sum_{t=s}^{L-2}
\binom{2s}{s} \binom{2t-2s+1}{t-s}  \dfrac{1}{4^t} \cdot y^L
\stackrel{(A[1])\rightarrow (A[0])}{=} \dfrac{y(1-\sqrt{1-y})}{4(1-y)^2}, \\
F_{[1,-1]}(y) &= \dfrac{1}{8} \sum_{L=0}^{\infty}\sum_{s=0}^{L-1}\sum_{t=s}^{L-2}
\binom{2s+1}{s} \binom{2t-2s}{t-s-1}  \dfrac{1}{4^t} \cdot y^L 
\stackrel{(A[-1])\rightarrow (A[1])}{=} \dfrac{(1-\sqrt{1-y})^3}{4(1-y)^2}, \\
F_{[1,1]}(y) &= \dfrac{1}{16} \sum_{L=0}^{\infty}\sum_{s=0}^{L-1}\sum_{t=s}^{L-2}
\binom{2s+1}{s} \binom{2t-2s+1}{t-s}  \dfrac{1}{4^t} \cdot y^L 
\stackrel{(A[1])\text{ twice}}{=} \dfrac{(1-\sqrt{1-y})^2}{4(1-y)^2}.
\end{align*}

Combining the four cases, the generating function $F^{(2)}(y)$ is simplified to
\begin{align*} F^{(2)}(y) &= F_{[0,0]}(y)+F_{[0,1]}(y)+F_{[1,-1]}(y)+F_{[1,1]}(y) \\
&= \dfrac{3-y}{2(1-y)^2}-\dfrac{3}{2(1-y)^{3/2}}.
\end{align*}
Again, the closed-form formula of the original sum is due to Proposition \ref{prop:expansion}.

The reader may notice that after applying the summation identities,
each of the four parts of $F^{(2)}(y)$ always takes the form $\dfrac{y^A(1-\sqrt{1-y})^B}{2^2(1-y)^2}$, for some non-negative integers $A$ and $B$. This observation leads to the following lemma which is used as a stepping stone to prove our main theorem.

\begin{lem}
\label{lem:partial}
Define
\begin{small}
\begin{equation}
\label{eq:gen_originalsum}
  F^{(r)}(y):=\sum_{L=0}^{\infty}\left(\sum_{1\leq i_1 < i_2 < \dots < i_r \leq 2L}  
\binom{i_1-1}{\lf i_1/2\rf}  \binom{i_2-i_1-1}{\lf i_2/2\rf- \lf i_1/2 \rf -1} \dots
\binom{i_r-i_{r-1}-1}{\lf i_r/2\rf- \lf i_{r-1}/2 \rf -1}  \dfrac{1}{2^{i_r}}  \right)y^L.   
\end{equation}
\end{small}

Then, the following holds. 

(i) $F^{(r)}(y)$ can be expressed as
\begin{equation} \label{eq:Fry}
F^{(r)}(y) = \sum_{i=1}^{2^r}\dfrac{y^{A_i} (1-\sqrt{1-y})^{B_i}}{2^r(1-y)^{1+r/2}},    
\end{equation} 
where $A_i$ is 0 or 1. Moreover, the non-negative integer $B_i$ is such that $B_i\leq r+1$ if $A_i=0$ or $B_i\leq r-1$ if $A_i=1$.

(ii) The partial fraction expansion form of  \eqref{eq:Fry}  is
\begin{equation} 
\label{eq:partial}
F^{(r)}(y) =   \dfrac{P(y)}{(1-y)^{1+r/2}}+ \dfrac{Q(y)}{(1-y)^{(1+r)/2}},  
\end{equation}
for some polynomials $P(y)$ and $Q(y)$, each having degree at most $\lc \dfrac{r}{2} \rc$
and $\lf \dfrac{r}{2} \rf$, respectively.
\end{lem}

\begin{proof}
For claim (i), Appendix \ref{sec:appB} provides a detailed calculation for the case $r=3$. Once familiar with the derivation procedure therein, the algebraic machinery behind the proof of this claim for general $r$ is almost routine. In fact it is easy to establish the following facts:\vspace{-1em}

\begin{enumerate}
\item The denominator is of the form $2^r(1-y)^{1+r/2}$;\vspace{-0.5em}
\item For each $i$, $A_i \in \{0,1\}$;\vspace{-0.5em}
\item For each $i$, if $A_i=0$, then $B_i \leq r+1$. Otherwise, if $A_i=1$, then $B_i \leq r-1.$
\end{enumerate}
\vspace{-1em}

The proofs of these facts rely heavily on the three summation identities in Proposition \ref{prop:three_identities}. We encourage the reader to take out a pen and paper and try to do the calculations on their fingers. The key points used to derive these facts are now summarized. 
\vspace{-0.5em}

For 1: Interchanging the order of the sum in \eqref{eq:gen_originalsum}, the inner sum (w.r.t.~the index $L$) contributes to the term $\dfrac{1}{1-y}$. Continuing working from the innermost toward the outermost sums, each sum contributes to the additional factor of $\dfrac{1}{\sqrt{1-y}}$ (by Proposition \ref{prop:three_identities}).
\vspace{-0.5em}

For 2: This depends on the parity of the index $i_1$ of the outermost sum in \eqref{eq:gen_originalsum}. If $i_1$ is odd then $A_i=1$, else $A_i=0$.
\vspace{-0.5em}

For 3: The argument relies on the number of times $(A[1])$ and $(A[-1])$ have been applied. Each time $(A[1])$ is applied, the degree of $1-\sqrt{1-y}$ increases by 1. Each time $(A[-1])$ is applied, the degree of $1-\sqrt{1-y}$ increases by 2. (The latter case does not happen too often.) 

As for claim (ii), the result follows immediately from the fact that degree of the numerator in \eqref{eq:Fry} is at most $\dfrac{r+1}{2}.$
\end{proof}

The following lemma and proposition turn the obtained generating functions into Theorem~\ref{Main}.

\begin{lem}
Let $c$ and $k$ be non-negative integers such that $0 \leq c \leq k$. Then,
\[   \binom{L+c-\frac{1}{2}}{k-\frac{1}{2}} 
=  \dfrac{1}{4^L}\binom{2L}{L} R(L), \]
where $R(L)$ is a polynomial in $L$ of degree $k$.
\end{lem}

\begin{proof}
We leave it as an
exercise for the reader to show that
\[   \binom{L+c-\frac{1}{2}}{k-\frac{1}{2}} 
=  \dfrac{1}{4^L}\binom{2L}{L} \left[ \binom{L}{k-c}\dfrac{\binom{2L+2c-1}{2L}}{\binom{L+c-1}{L}}
\cdot \dfrac{\binom{2c-1}{c}}{\binom{2k-1}{k}}\dfrac{4^{k-c}}{\binom{k}{c}}\right]. \]
Then, $\binom{L}{k-c}$ is a polynomial of degree $k-c$ and 
$\dfrac{\binom{2L+2c-1}{2L}}{\binom{L+c-1}{L}}$ can be simplified to a polynomial
of degree $(2c-1)-(c-1) = c.$ The last portion in the bracket is just a constant in $L$.
Thus, the entire bracket is a polynomial of degree~$(k-c)+c=k$. 
\end{proof}

The closed-form formula of the original sum $\dfrac{1}{2^n}\sum_{1\leq i_1 < i_2 < \dots < i_r \leq h}C[Y_{i_1}Y_{i_2}\dots Y_{i_r}]$ is the coefficient of $y^L$ in $F^{(r)}(y)$. Expression \eqref{eq:partial} of Lemma \ref{lem:partial} tells us that $F^{(r)}(y)$ can be decomposed into two fractions in which the denominators have integer exponent and fractional exponent. The proposition below is used to unmask the coefficient of $y^L$ in each fraction, and eventually allows us to establish the claim about the closed-form formula of the original sum in \eqref{eq:sum_CYr}. 

In what follows, we denote by $[y^L]f(y)$ the coefficient of~$y^L$ in $f(y)$.

\begin{prop}
\label{prop:expansion}
Let $k$ be a non-negative integer and $H(y)$ a polynomial of degree~$k$. 

(i) The integer exponent:
\[  [y^L] \dfrac{H(y)}{(1-y)^{k+1}} = Q(L), \]
where $Q(L)$ is a polynomial of degree at most $k$.

(ii) The fractional exponent:   
\[ [y^L] \dfrac{H(y)}{(1-y)^{k+1/2}} = \dfrac{1}{4^L}\binom{2L}{L}P(L), \]
where $P(L)$ is a polynomial of degree at most $k.$
\end{prop}

\begin{proof}
Assume $H(y) = \sum_{i=0}^k  a_iy^i$.

(i) Since $ \dfrac{1}{(1-y)^{k+1}} = \sum_L \binom{L+k}{k}y^L,$
we have 
\begin{small}
\[  \dfrac{H(y)}{(1-y)^{k+1}} 
= \sum_{i=0}^k  a_i \sum_{L=0}^{\infty} \binom{L+k}{k}y^{L+i}
= \sum_{i=0}^k  a_i \sum_{L=i}^{\infty} \binom{L-i+k}{k}y^{L}
= \sum_{L=0}^{\infty} \left[ \sum_{i=0}^{\min{(L,k)}}  a_i  \binom{L-i+k}{k} \right] y^{L}. 
\]
\end{small}

The sum inside the bracket is a polynomial (in $L$) of degree at most $k$ as claimed.

(ii) Since $ \dfrac{1}{(1-y)^{k+1/2}} = \sum_L \binom{L+k-\frac{1}{2}}{k-\frac{1}{2}}y^L,$
we have 
\begin{align*}  \dfrac{H(y)}{(1-y)^{k+1/2}} 
&= \sum_{i=0}^k  a_i \sum_{L=0}^{\infty} \binom{L+k-\frac{1}{2}}{k-\frac{1}{2}}y^{L+i} 
= \sum_{i=0}^k  a_i \sum_{L=i}^{\infty} \binom{L-i+k-\frac{1}{2}}{k-\frac{1}{2}} y^{L} \\
&= \sum_{L=0}^{\infty} \left[ \sum_{i=0}^{\min{(L,k)}}  a_i  
\binom{L-i+k-\frac{1}{2}}{k-\frac{1}{2}} \right] y^{L} .
\end{align*}
Applying the previous lemma to each binomial term of the inner sum, the claim holds.
\end{proof}

\section{Detailed derivations of Lemma \ref{lem:partial}(i) for $r=3$\label{sec:appB}}
Consider the original sum
\begin{equation}
\label{eq:inner_r3}
  \sum_{1\leq i < j < k \leq 2L}  
\binom{i-1}{\lf i/2\rf}  \binom{j-i-1}{\lf j/2\rf- \lf i/2 \rf -1} \binom{k-j-1}{\lf k/2\rf- \lf j/2 \rf -1}  \dfrac{1}{2^{k}}.   
\end{equation}

\begin{normalsize}
\begin{table}[h!]
\caption{The original sum decomposition, the corresponding label, the exponents $A, B$, and the compensation factor for the generating function in each case\label{tab:decompose_r3}}
\centering{}%
\begin{tabular}{|c|c|c|c|l|c|c|c|}
\hline 
\multirow{2}{*}{{\small{}Case}} & \multirow{2}{*}{{\small{}$i$}} & \multirow{2}{*}{{\small{}$j$}} & \multirow{2}{*}{{\small{}$k$}} & \multirow{2}{*}{{\small{}Label}} & \multirow{2}{*}{{\small{}$A$}} & \multirow{2}{*}{{\small{}$B$}} & {\small{}Compensation}\tabularnewline
 &  &  &  &  &  &  & {\small{}factor}\tabularnewline
\hline 
{\small{}1} & {\small{}$2s+1$} & {\small{}$2t+2$} & {\small{}$2r+3$} & {\small{}$[0,0,-1]$} & {\small{}$1$} & {\small{}$2$} & {\small{}$y^{2}$}\tabularnewline
\hline 
{\small{}2} & {\small{}$2s+1$} & {\small{}$2t+2$} & {\small{}$2r+4$} & {\small{}$[0,0,1]$} & {\small{}$1$} & {\small{}$1$} & {\small{}$y^{2}$}\tabularnewline
\hline 
{\small{}3} & {\small{}$2s+1$} & {\small{}$2t+3$} & {\small{}$2r+4$} & {\small{}$[0,1,0]$} & {\small{}$1$} & {\small{}$1$} & {\small{}$y^{2}$}\tabularnewline
\hline 
{\small{}4} & {\small{}$2s+1$} & {\small{}$2t+3$} & {\small{}$2r+5$} & {\small{}$[0,1,1]$} & {\small{}$1$} & {\small{}$2$} & {\small{}$y^{3}$}\tabularnewline
\hline 
{\small{}5} & {\small{}$2s+2$} & {\small{}$2t+3$} & {\small{}$2r+4$} & {\small{}$[1,-1,0]$} & {\small{}$0$} & {\small{}$3$} & {\small{}$y^{2}$}\tabularnewline
\hline 
{\small{}6} & {\small{}$2s+2$} & {\small{}$2t+3$} & {\small{}$2r+5$} & {\small{}$[1,-1,1]$} & {\small{}$0$} & {\small{}$4$} & {\small{}$y^{3}$}\tabularnewline
\hline 
{\small{}7} & {\small{}$2s+2$} & {\small{}$2t+4$} & {\small{}$2r+5$} & {\small{}$[1,1,-1]$} & {\small{}$0$} & {\small{}$4$} & {\small{}$y^{3}$}\tabularnewline
\hline 
{\small{}8} & {\small{}$2s+2$} & {\small{}$2t+4$} & {\small{}$2r+6$} & {\small{}$[1,1,1]$} & {\small{}$0$} & {\small{}$3$} & {\small{}$y^{3}$}\tabularnewline
\hline
\end{tabular}
\end{table}
\end{normalsize}

Depending on the parity of the indices $i, j, k$, the original sum in \eqref{eq:inner_r3} can be decomposed into eight cases as shown in Table \ref{tab:decompose_r3}. The column ``Label'' specifies the order in which the summation identities $(A[0])$, $(A[1])$, or $(A[-1])$ of Proposition \ref{prop:three_identities} must be applied when evaluating the nested sum. Recall that the label is read and applied in order from right to left. Our goal is to compute the generating function of \eqref{eq:inner_r3}, i.e. $F^{(3)}(y)$ which can be done by applying the generating function and evaluating the nested sum of each case separately. We look at a couple of examples now.\\* 

For case 1 ($i=2s+1$,~$j=2t+2$,~$k=2r+3$),
\begin{footnotesize}
\begin{align*}
\sum_{L=0}^{\infty}\left(\sum_{s=0}^{L-1}  \sum_{t=s}^{L-1}  \sum_{r=t}^{L-2} 
\binom{2s}{s}  \binom{2t-2s}{t-s} \binom{2r-2t}{r-t -1}  \dfrac{1}{2^{2r+3}}\right)y^L 
&= \frac{1}{8}\cdot \sum_{0\leq s\leq t\leq r\leq L-2} 
\binom{2s}{s}  \binom{2t-2s}{t-s} \binom{2r-2t}{r-t -1}  \dfrac{1}{4^{r}}\sum_{L=r+2}^{\infty}y^L  \\ 
&= \frac{1}{8}\cdot \frac{y^2}{1-y} \sum_{0\leq s\leq t} 
\binom{2s}{s}  \binom{2t-2s}{t-s} \underbrace{\sum_{r=t}^{\infty} \binom{2r-2t}{r-t -1} \left(\frac{y}{4}\right)^r}_{(A[-1])}\\
&=\frac{1}{8}\cdot \frac{y^2}{1-y} \sum_{0\leq s\leq t} 
\binom{2s}{s}  \binom{2t-2s}{t-s} \frac{\left(1-\sqrt{1-y}\right)^2}{y\sqrt{1-y}}\left(\frac{y}{4}\right)^t\\
&=\frac{1}{8}\cdot \frac{y^2}{1-y} \frac{\left(1-\sqrt{1-y}\right)^2}{y\sqrt{1-y}}\sum_{0\leq s} \binom{2s}{s} \underbrace{\sum_{t=s}^{\infty} \binom{2t-2s}{t-s} \left(\frac{y}{4}\right)^{t-s}}_{(A[0])} \left(\frac{y}{4}\right)^s\\
&=\frac{1}{8}\cdot \frac{y^2}{1-y} \frac{\left(1-\sqrt{1-y}\right)^2}{y\sqrt{1-y}}\frac{1}{\sqrt{1-y}}
\underbrace{\sum_{s=0}^\infty \binom{2s}{s} \left(\frac{y}{4}\right)^s}_{(A[0])}\\
&=\frac{1}{8}\cdot \frac{y^2}{1-y} \frac{\left(1-\sqrt{1-y}\right)^2}{y\sqrt{1-y}}\frac{1}{\sqrt{1-y}}\frac{1}{\sqrt{1-y}}\\ 
&= \frac{y\left(1-\sqrt{1-y}\right)^2}{2^3\left(1-y\right)^{1+3/2}}.
\end{align*}
\end{footnotesize}

We can see that case 1 above leads to the exponents $A=1$ and $B=2$, and the denominator satisfies what is stated in \eqref{eq:Fry}. Taking a closer look at how we calculated the above sum, we arrive at the following shortcut method. 
\begin{footnotesize}
\[
\sum_{L=0}^{\infty}\left(\sum_{s=0}^{L-1}  \sum_{t=s}^{L-1}  \sum_{r=t}^{\boxed{L-2}} 
\underbrace{\binom{2s}{s}}_{\circled{IV}(A[0])} \underbrace{\binom{2t-2s}{t-s}}_{\circled{III}(A[0])}
\underbrace{\binom{2r-2t}{r-t -1}}_{\circled{II}(A[-1])}
\dfrac{1}{2^{2r+3}}\right) \underbrace{y^L}_{\circled{I}}
=\frac{1}{8}\cdot \frac{\boxed{y^2}}{\underbrace{1-y}_{\circled{I}}}
\underbrace{\frac{\left(1-\sqrt{1-y}\right)^2}{y\sqrt{1-y}}}_{\circled{II}}
\underbrace{\frac{1}{\sqrt{1-y}}}_{\circled{III}}
\underbrace{\frac{1}{\sqrt{1-y}}}_{\circled{IV}}.
\]
\end{footnotesize}
\begin{enumerate}
    \item The compensation factor $y^2$ (the term in the box) is obtained by looking at the upper bound $L-2$ of the sum index $r$. The bounds of the sum indices $s$ and $t$ are irrelevant. \vspace{-0.5em}
    \item The term $\frac{1}{1-y}$ is due to the sum with respect to the index $L$. \vspace{-0.5em}
    \item By the summation identities, each of the sums with respect to indices $r,t,s$ contributes to the additional factor of $\dfrac{1}{\sqrt{1-y}}$. \vspace{-0.5em}
    \item The initial compensation factor $y^2$ in the numerator is cancelled with the factor $y$ in the denominator due to $(A[-1])$.  \vspace{-0.5em}
\end{enumerate}

Let us use the shortcut method to evaluate case 7 ($i=2s+2$, $j=2t+4$, $k=2r+5$).
\begin{footnotesize}
\begin{align*}
\sum_{L=0}^{\infty}&\left(\sum_{s=0}^{L-1}  \sum_{t=s}^{L-2}  \sum_{r=t}^{\boxed{L-3}} 
\underbrace{\binom{2s+1}{s}}_{\circled{IV}(A[1])} \underbrace{\binom{2t-2s+1}{t-s}}_{\circled{III}(A[1])}
\underbrace{\binom{2r-2t}{r-t -1}}_{\circled{II}(A[-1])}
\dfrac{1}{2^{2r+5}}\right) \underbrace{y^L}_{\circled{I}}\\
&=\frac{1}{32}\cdot \frac{\boxed{y^3}}{\underbrace{1-y}_{\circled{I}}}
\underbrace{\frac{\left(1-\sqrt{1-y}\right)^2}{y\sqrt{1-y}}}_{\circled{II}}
\underbrace{\frac{2\left(1-\sqrt{1-y}\right)}{\sqrt{1-y}}}_{\circled{III}}
\underbrace{\frac{2\left(1-\sqrt{1-y}\right)}{\sqrt{1-y}}}_{\circled{IV}}
= \frac{y^0\left(1-\sqrt{1-y}\right)^4}{2^3\left(1-y\right)^{1+3/2}}.
\end{align*}
\end{footnotesize}

Thus, the exponents for this case are $A=0$ and $B=4$, respectively. We further observe that a fraction is put in lowest terms by cancelling out the common factor of the the initial denominator $32$ with the numerator $2\cdot2$. The only way to obtain the factor $2$ in the numerator is when the identity $(A[1])$ is applied. The reader is encouraged to verify that all these cases have the same factor $\frac{1}{8}$ in the final results. Table \ref{tab:decompose_r3} also provides full information of the label, the exponents $A, B$ as well as the compensation factor of the final generating function for all cases. Note also that the value of the exponent $A$ only depends on the parity of the index $i$ w.r.t. the outermost sum in \eqref{eq:inner_r3}. If $i$ is odd then $A=1$, else $A=0$.

Using the above analysis, the three facts we stated in the proof of Lemma \ref{lem:partial}(i) for general $F^{(r)}(y)$ can be verified easily. 

\end{document}